\documentclass[11pt]{amsart}
\usepackage{pdfsync}
\usepackage{latexsym}
\usepackage{amsgen,amsmath,amsxtra,amssymb,amsfonts,amsthm}
\usepackage{times}
\usepackage{graphicx}
\usepackage[dvips]{epsfig}
\usepackage{subfigure}
\usepackage[active]{srcltx}
\usepackage{color}

\begin{document}
\newcommand{\dyle}{\displaystyle}
\newcommand{\R}{{\mathbb{R}}}
 \newcommand{\Hi}{{\mathbb H}}
\newcommand{\Ss}{{\mathbb S}}
\newcommand{\N}{{\mathbb N}}
\newcommand{\Rn}{{\mathbb{R}^n}}
\newcommand{\ieq}{\begin{equation}}
\newcommand{\eeq}{\end{equation}}
\newcommand{\ieqa}{\begin{eqnarray}}
\newcommand{\eeqa}{\end{eqnarray}}
\newcommand{\ieqas}{\begin{eqnarray*}}
\newcommand{\eeqas}{\end{eqnarray*}}
\newcommand{\Bo}{\put(260,0){\rule{2mm}{2mm}}\\}
\def\L#1{\label{#1}} \def\R#1{{\rm (\ref{#1})}}


\theoremstyle{plain}
\newtheorem{theorem}{Theorem} [section]
\newtheorem{corollary}[theorem]{Corollary}
\newtheorem{lemma}[theorem]{Lemma}
\newtheorem{proposition}[theorem]{Proposition}
\def\neweq#1{\begin{equation}\label{#1}}
\def\endeq{\end{equation}}
\def\eq#1{(\ref{#1})}

\theoremstyle{definition}
\newtheorem{definition}[theorem]{Definition}
\newtheorem{remark}[theorem]{Remark}

\numberwithin{figure}{section}
\newcommand{\res}{\mathop{\hbox{\vrule height 7pt width .5pt depth
0pt \vrule height .5pt width 6pt depth 0pt}}\nolimits}
\def\at#1{{\bf #1}: } \def\att#1#2{{\bf #1}, {\bf #2}: }
\def\attt#1#2#3{{\bf #1}, {\bf #2}, {\bf #3}: } \def\atttt#1#2#3#4{{\bf #1}, {\bf #2}, {\bf #3},{\bf #4}: }
\def\aug#1#2{\frac{\displaystyle #1}{\displaystyle #2}} \def\figura#1#2{ \begin{figure}[ht] \vspace{#1} \caption{#2}
\end{figure}} \def\B#1{\bibitem{#1}} \def\q{\int_{\Omega^\sharp}}
\def\z{\int_{B_{\bar{\rho}}}\underline{\nu}\nabla (w+K_{c})\cdot
\nabla h} \def\a{\int_{B_{\bar{\rho}}}}
\def\b{\cdot\aug{x}{\|x\|}}
\def\n{\underline{\nu}} \def\d{\int_{B_{r}}}
\def\e{\int_{B_{\rho_{j}}}} \def\LL{{\mathcal L}}
\def\itr{\mathrm{Int}\,}
\def\D{{\mathcal D}}
 \def\tg{\tilde{g}}
\def\A{{\mathcal A}}
\def\S{{\mathcal S}}
\def\H{{\mathcal H}}
\def\M{{\mathcal M}}
\def\T{{\mathcal T}}
\def\U{{\mathcal U}}
\def\N{{\mathcal N}}
\def\I{{\mathcal I}}
\def\F{{\mathcal F}}
\def\J{{\mathcal J}}
\def\E{{\mathcal E}}
\def\F{{\mathcal F}}
\def\G{{\mathcal G}}
\def\HH{{\mathcal H}}
\def\W{{\mathcal W}}
\def\H{\D^{2*}_{X}}
\def\d{d^X_M }
\def\LL{{\mathcal L}}
\def\H{{\mathcal H}}
\def\HH{{\mathcal H}}
\def\itr{\mathrm{Int}\,}
\def\vah{\mbox{var}_\Hi}
\def\vahh{\mbox{var}_\Hi^1}
\def\vax{\mbox{var}_X^1}
\def\va{\mbox{var}}
\def\SS{{\mathcal S}}
 \def\Y{{\mathcal Y}}
\def\length{{l_\Hi}}
\newcommand{\average}{{\mathchoice {\kern1ex\vcenter{\hrule
height.4pt width 6pt depth0pt} \kern-11pt} {\kern1ex\vcenter{\hrule height.4pt width 4.3pt depth0pt} \kern-7pt} {} {} }}
\def\weak{\rightharpoonup}
\def\detu{{\rm det}(D^2u)}
\def\detut{{\rm det}(D^2u(t))}
\def\detvt{{\rm det}(D^2v(t))}
\def\detv{{\rm det}(D^2v)}
\def\uuu{u_xu_yu_{xy}}
\def\uuut{u_x(t)u_y(t)u_{xy}(t)}
\def\uuus{u_x(s)u_y(s)u_{xy}(s)}
\def\uuutn{u_x(t_n)u_y(t_n)u_{xy}(t_n)}
\def\vvv{v_xv_yv_{xy}}
\newcommand{\ave}{\average\int}

\title[A 4th order parabolic PDE involving the Hessian]{Global existence versus blow-up results
for a fourth order parabolic PDE involving  the Hessian}

\author[C. Escudero, F. Gazzola, I. Peral]{Carlos Escudero, Filippo Gazzola, Ireneo Peral}
\address{}
\email{Corresponding author: Filippo Gazzola, filippo.gazzola@polimi.it}


\date{\today}

\begin{abstract}
We consider a partial differential equation that arises in the coarse-grained description of epitaxial growth
processes. This is a parabolic equation whose evolution is governed by the competition between the determinant
of the Hessian matrix of the solution and the biharmonic operator.
This model might present a gradient flow structure depending on the boundary conditions.
We first extend previous results on the existence of stationary solutions to this model for Dirichlet boundary conditions.
For the evolution problem we prove local existence of solutions
for arbitrary data and global existence of solutions for small data. By exploiting the boundary conditions and
the variational structure of the equation, according to the size of the data we prove finite time blow-up of the solution and/or
convergence to a stationary solution for global solutions.\par\noindent
{\bf R\'esum\'e.} On consid\`ere une \'equation diff\'erentielle qui d\'ecrit la croissance \'epitaxiale d'une couche rugueuse de fa\c{c}on macroscopique.
Il s'agit d'une \'equation parabolique pour laquelle l'\'evolution est gouvern\'ee par une comp\'etiton entre le d\'eterminant Hessien
de la solution et l'op\'erateur biharmonique. Ce mod\`ele peut pr\'esenter une structure de flux gradient suivant les conditions au bord.
On \'etend d'abord des r\'esultats pr\'ec\'edents sur l'existence de solutions stationnaires pour ce mod\`ele avec des conditions de Dirichlet.
Pour l'\'equation d'\'evolution on prouve l'existence locale de solutions pour tout donn\'e initial et l'existence globale pour des donn\'es
suffisamment petits. En exploitant les conditions au bord et la structure variationnelle de l'\'equation, suivant la taille du donn\'e initial on
d\'emontre l'explosion en temps fini et/ou la convergence \`a une solution stationnaire pour les solutions globales.
\end{abstract}
\renewcommand{\thefootnote}{\fnsymbol{footnote}}
\setcounter{footnote}{-1}
\footnote{Supported by MTM2010-18128, RYC-2011-09025, MIUR-PRIN-2008.}
\renewcommand{\thefootnote}{\arabic{footnote}}
\maketitle

\section{Introduction}

Epitaxial growth is a technique by means of which the deposition of new
material on existing layers of the same material takes place under high vacuum
conditions. It is used in the semiconductor industry
for the growth of crystalline structures that
might be composed of pure chemical elements like silicon or
germanium, or it could instead be formed by alloys like gallium arsenide or
indium phosphide. In the case of molecular beam epitaxy the deposition
is a very slow process and happens almost atom by atom.

Throughout this paper we assume that $\Omega \subset \mathbb{R}^2$ is an open, bounded smooth domain which is the place where
the deposition takes place. Although this kind of mathematical model can be studied in
any spatial dimension $N$, we will concentrate here on the
physical situation $N=2$. The macroscopic evolution of the growth process can be modeled with a partial
differential equation that is frequently proposed invoking
phenomenological and symmetry arguments~\cite{barabasi,marsili}.
The solution of such a differential equation is the function
\begin{equation}\nonumber
u: \Omega \times \mathbb{R}_+ \rightarrow \mathbb{R},
\end{equation}
describing the height of the growing interface at the spatial
location $x \in \Omega$ at the temporal instant $t \in\mathbb{R}_+:=[0,\infty)$. A fundamental modeling assumption in this field is considering
that the physical interface can be described as the graph of $u$, and this is a valid hypothesis
in an important number of cases~\cite{barabasi}.

One of the most widespread examples of this type of theory is the Kardar-Parisi-Zhang equation~\cite{kpz}
\begin{equation}\nonumber
u_t = \nu \Delta u + \gamma |\nabla u|^2 + \eta(x,t),
\end{equation}
which has been extensively studied in the physical literature and
has also been investigated for its interesting
mathematical properties~\cite{ireneo1,ireneo2,BSW,BMP,GGK}. On the other hand, it has been argued
that epitaxial growth processes should be described by
a different equation coming from a conservation law and, in particular,
the term $|\nabla u|^2$ should not be present in such a model~\cite{barabasi}. An equation
fulfilling these properties is the conservative counterpart of the Kardar-Parisi-Zhang
equation~\cite{vlds,sun,villain}
\begin{equation}
\label{ssg} u_t = -\mu \Delta^2 u + \kappa \Delta |\nabla u|^2 +
\zeta(x,t).
\end{equation}
This equation is conservative in the sense that the mean value
$\int_\Omega u \, dx$ is constant if boundary conditions
that isolate the system are used. It can also be considered as a higher order
counterpart of the Kardar-Parisi-Zhang equation. In recent years, much attention has been devoted to other models of epitaxial growth,
see \cite{grasselli,kohn,kohn1,li, winkler} and references therein.

Herein we will consider a different model obtained by means of the variational formulation
developed in~\cite{marsili} and aimed at unifying previous approaches. We skip the detailed derivation
of our model, that can be found in~\cite{n2}, and move to the resulting equation, that reads
\begin{equation}
\nonumber u_t = 2 \, K_1 \, \det \left( D^2 u \right) -
K_2 \, \Delta^2 u + \xi(x,t).
\end{equation}
This partial differential equation can be thought of as an
analogue of equation~(\ref{ssg}); in fact, they are identical from a strict
dimensional analysis viewpoint. Let us also note that this model has been shown
to constitute a suitable description of epitaxial growth in the same sense as equation (\ref{ssg}), and it even
displays more intuitive geometric properties~\cite{escudero,elka}.
The constants $K_1$ and $K_2$ will be rescaled in the following.

In this work we are interested in the following initial-boundary value problem:
\begin{equation}\label{parabolicpdef}
\left\{\begin{array}{ll}
u_t + \Delta^2 u = \det(D^2 u) + \lambda f \qquad & x \in \Omega\, ,\ t>0\, ,\\
u(x,0)=u_0(x), \qquad & x \in \Omega\, ,\\
\text{boundary conditions}\qquad & x \in \partial\Omega\, ,\ t>0\, ,
\end{array}\right.\end{equation}
where $f$ is some function possibly depending on both space and time coordinates and belonging to some Lebesgue space, $\lambda \in \mathbb{R}$.
The initial condition $u_0(x)$ is also assumed to belong to some Sobolev space.
We will consider the following sets of boundary conditions
\begin{equation}\nonumber
u = u_\nu = 0, \qquad x \in \partial \Omega,
\end{equation}
which we will refer to as Dirichlet boundary conditions, and
\begin{equation}\nonumber
u = \Delta u = 0, \qquad x \in \partial \Omega,
\end{equation}
which we will refer to as Navier boundary conditions.
We note that the stationary solutions to this model were studied before~\cite{n2,n3,n1}.\par
For the evolution problem \eq{parabolicpdef} we prove existence of a solution, both for arbitrary time intervals and small data, and for arbitrary
data and small time intervals. Then using several tools from both critical point theory and potential well techniques, we prove the existence of
finite time blow up solutions as well as the existence of global in time solutions, in suitable functional spaces. The use of these tools is by
far nontrivial both because the nonlinearity occurs in the second order derivatives and because more regularity is necessary to overcome some
delicate technical points.\par
This paper is organized as follows. In Section \ref{stationaryproblem} we extend previous results in~\cite{n1} concerning the stationary problem
with Dirichlet boundary conditions and characterize the geometry of the functional that allows the variational
treatment of this problem. In Section \ref{parabolicwithsource} we build the existence theory for the parabolic problem with both sets of boundary
conditions and the presence of a source term.
Section \ref{parabolicwithoutsource} is devoted to the analysis of the long time behavior and the blow-up in finite time of
the solutions to the Dirichlet problem in the absence of a source term;
this analysis is carried out taking advantage of the gradient flow structure of the equation in this case and of the so-called potential well techniques.
Finally, in Section \ref{open} we present some further results, including the proof of finite time blow-up of the solutions to the Navier problem for large enough
initial conditions, and propose some open questions.

\section{The stationary problem}\label{stationaryproblem}

\subsection{Existence of solutions with Dirichlet conditions.}
In the sequel, we need several different norms. All the norms in $W^{s,p}$-spaces will be reported explicitly (that is, $\|\cdot\|_{W^{s,p}(\Omega)}$) except
for the $L^p$-norm and the $W^{2,2}_0$-norm, respectively denoted by
$$(1\le p<\infty)\quad \|u\|_p^p=\int_\Omega|u|^p\ ,\qquad\|u\|_\infty={\rm ess}\sup_{x\in\Omega}|u(x)|\ ,$$
$$\|u\|^2=\|\Delta u\|_2^2=\int_\Omega|\Delta u|^2\ ,$$

We start by focusing on the following nonhomogeneous problem
\begin{equation}\label{nonhomogeneous}
\left\{
\begin{array}{ll}
\Delta^2 u = \det (D^2 u) + f \quad & \mbox{in }\Omega\\
u=g \quad & \mbox{on }\partial\Omega \\
u_\nu = h \quad & \mbox{on }\partial\Omega
\end{array}
\right. ,
\end{equation}
where $f \in L^1(\Omega)$, $g\in W^{\frac32,2}(\partial\Omega)$, $h\in W^{\frac12,2}(\partial\Omega)$. The following result holds.

\begin{theorem}\label{nonhDir}
There exists $\gamma>0$ such that if
\begin{equation}\label{gamma}
\|f\|_1+\|g\|_{W^{3/2,2}(\partial\Omega)}+\|h\|_{W^{1/2,2}(\partial\Omega)}<\gamma
\end{equation}
then \eqref{nonhomogeneous} admits at least two weak solutions in $W^{2,2}(\Omega)$, a stable solution and a mountain pass solution.
\end{theorem}
\begin{proof} Consider the auxiliary linear problem
\begin{equation}\label{v}
\left\{
\begin{array}{ll}
\Delta^2 v = f \quad & \mbox{in }\Omega\\
v=g \quad & \mbox{on }\partial\Omega\\
v_\nu = h\quad & \mbox{on }\partial\Omega
\end{array}
\right. .
\end{equation}
In view of the embedding $L^1(\Omega)\subset W^{-2,2}(\Omega)$, \cite[Theorem 2.16]{ggs} tells us that \eqref{v} admits a unique weak solution $v\in W^{2,2}(\Omega)$
which satisfies
\begin{equation}\label{estimate}
\|D^2 v\|_2\le C\Big(\|f\|_1+\|g\|_{W^{3/2,2}(\partial\Omega)}+\|h\|_{W^{1/2,2}(\partial\Omega)}\Big)
\end{equation}
for some $C>0$ independent of $f$, $g$, $h$. Subtracting \eqref{v} from \eqref{nonhomogeneous} and putting $w=u-v$ we get
$$
\left\{
\begin{array}{ll}
\Delta^2 w = \det [D^2 (w + v)] \quad & \mbox{in }\Omega\\
w=w_\nu = 0\quad & \mbox{on }\partial\Omega
\end{array}
\right. .
$$
This problem can be written as
\begin{equation}\label{w1}
\left\{
\begin{array}{lll}
\Delta^2 w\! =\! \det (D^2 w)\! +\! \det (D^2 v)\! +\! v_{xx} w_{yy}\! +\! w_{xx} v_{yy}\! -\! 2 w_{xy} v_{xy}\ & \mbox{in }\Omega \\
w=w_\nu = 0\ & \mbox{on }\partial\Omega
\end{array}
\right. .
\end{equation}

By combining results from \cite{coifman,D,muller}, Escudero-Peral \cite{n1} proved that for all $u\in W^{2,2}_0(\Omega)$ one has
that $\detu$ belongs to the Hardy space and that
$$\detu=\Big(u_xu_{yy}\Big)_x-\Big(u_xu_{xy}\Big)_y=\Big(u_xu_y\Big)_{xy}-\frac{1}{2}\Big(u_y^2\Big)_{xx}-\frac{1}{2}\Big(u_x^2\Big)_{yy}$$
in $\mathcal{D}'(\Omega)$. Moreover,
\begin{equation}\label{vdetv}
\int_\Omega u\, \detu=3\int_\Omega\uuu\quad \forall u\in W^{2,2}_0(\Omega).
\end{equation}
These facts show that \eqref{w1} admits a variational formulation. The corresponding functional reads
$$
K(w)=\int_\Omega\left[ \frac{|\Delta w|^2}{2}\!-\! w_x w_y w_{xy}\! -\! \det(D^2 v) w\! +\! \frac{w^2_y v_{xx}}{2}\! +\! \frac{w_x^2 v_{yy}}{2}\!
-\! w_x w_y v_{xy} \right].
$$
Note that, by the embedding $W^{2,2}_0(\Omega)\subset W^{1,4}_0(\Omega)$, we have
$$
K(w) \ge - \int_\Omega \left[ w_x w_y w_{xy} + \det(D^2 v) w \right] + \frac{1}{2} \|\Delta w\|_2^2 - C \|D^2v\|_2 \|\Delta w\|_2^2,
$$
so a mountain pass geometry \cite{ar} is ensured for small enough $\|D^2v\|_2$. In view of \eqref{estimate}, the mountain pass geometry is ensured if $\gamma$
in \eqref{gamma} is sufficiently small. This geometry yields the existence of a locally minimum solution and of a mountain pass solution.\end{proof}

Theorem \ref{nonhDir} generalizes the following statement proved in \cite{n1}:

\begin{corollary}
The Dirichlet problem
\begin{equation}\label{eqtnstat}
\left\{\begin{array}{ll}
\Delta^2u=\detu\quad & \mbox{in }\Omega\\
u=u_\nu=0\quad & \mbox{on }\partial\Omega
\end{array}\right.
\end{equation}
admits a nontrivial weak solution $u\in W^{2,2}_0(\Omega)$.
\end{corollary}

Concerning the regularity of solutions, we have the following statement.

\begin{theorem}\label{regularity}
Assume that, for some integer $k\ge0$ we have: $\partial\Omega\in C^{k+4}$, $f \in W^{k,2}(\Omega)$, $g\in W^{k+7/2,2}(\partial\Omega)$,
$h\in W^{k+5/2,2}(\partial\Omega)$. Then any solution to \eqref{nonhomogeneous} satisfies
$$u\in W^{k+4,2}(\Omega)\ .$$
In particular, any solution to \eqref{eqtnstat} is as smooth as the boundary permits.
\end{theorem}
\begin{proof} By duality, from the embedding $W^{s,2}_0(\Omega)\subset L^\infty(\Omega)$ we infer that
$L^1(\Omega)\subset[L^\infty(\Omega)]'\subset W^{-s,2}(\Omega)$ for all $s>1$. Therefore, for any solution $u\in W^{2,2}(\Omega)$ to \eqref{nonhomogeneous}
we have $\detu\in W^{-s,2}(\Omega)$ for all $s>1$. Therefore, even if $k=0$, we have $\Delta^2u\in W^{-s,2}(\Omega)$ and, in turn, $u\in W^{r,2}(\Omega)$
for any $r<3$. A bootstrap argument and elliptic regularity then allow to conclude.\end{proof}

\begin{remark} If we stop the previous proof at the first step, we see that, in a $C^3$ domain, any solution to
$$
\left\{\begin{array}{ll}
\Delta^2u=\detu +f\quad & \mbox{in }\Omega\\
u=u_\nu=0\quad & \mbox{on }\partial\Omega
\end{array}\right.
$$
with $f\in L^1(\Omega)$ belongs to $W^{r,2}(\Omega)$ for any $r<3$, which slightly improves the result in \cite{n1}. Note also that these arguments
take strong advantage of being in planar domains.\end{remark}

\subsection{The Nehari manifold and the mountain pass level}

The energy functional for the stationary problem (\ref{eqtnstat}) is
\begin{equation}\label{J}
J(v)=\frac{1}{2}\int_\Omega|\Delta v|^2-\int_\Omega\vvv\qquad\forall v\in W^{2,2}_0(\Omega) .
\end{equation}
It is shown in \cite{n1} that $J$ has a mountain pass geometry and that the corresponding mountain pass level is given by
\begin{equation}\label{d}
d\ =\ \inf_{\gamma\in\Gamma}\ \max_{0\le s\le1}\ J(\gamma(s))
\end{equation}
where $\Gamma:=\{\gamma\in C([0,1],W^{2,2}_0(\Omega));\, \gamma(0)=0,\, J(\gamma(1))<0\}$.
We aim to characterize differently $d$ and to relate it with the so-called Nehari manifold defined by
$$\mathcal{N}:=\Big\{v\in W^{2,2}_0(\Omega)\setminus\{0\};\, \langle J'(v),v\rangle=\|v\|^2-3\int_\Omega\vvv=0\Big\}$$
where $\langle\cdot,\cdot\rangle$ denotes the duality pairing between $W^{2,-2}(\Omega)$ and $W^{2,2}_0(\Omega)$.
To this end, we introduce the set
\begin{equation}\label{B}
B:=\{v\in  W^{2,2}_0(\Omega);\, \int_\Omega\vvv=1\}\, .
\end{equation}
It is clear that $v\in \mathcal{N}$ if and only if $\alpha v\in B$ for some $\alpha>0$.
In particular, not on all the straight directions starting from 0 in the
phase space $W^{2,2}_0(\Omega)$ there exists an intersection with $\mathcal{N}$. Hence, $\mathcal{N}$ is an unbounded manifold (of codimension 1) which separates the two regions
$$\mathcal{N}_+=\Big\{v\in W^{2,2}_0(\Omega);\, \|v\|^2>3\int_\Omega\vvv\Big\}$$ and
$$\mathcal{N}_-=\Big\{v\in W^{2,2}_0(\Omega);\, \|v\|^2<3\int_\Omega \vvv\Big\}\ .$$

The next result states some properties of $\mathcal{N}_\pm$.

\begin{theorem}\label{NN}
Let $v\in W^{2,2}_0(\Omega)$, then the following implications hold:\par
(i) $0<\|v\|^2<6d\ \Longrightarrow\ v\in \mathcal{N}_+$;\par
(ii) $v\in \mathcal{N}_+,\ J(v)<d\ \Longrightarrow\ 0<\|v\|^2<6d$;\par
(iii) $v\in\mathcal{ N}_- \Longrightarrow\ \|v\|^2 >6d$.
\end{theorem}
\begin{proof} It is well-known \cite{ar} that the mountain pass level $d$ may also be defined by
\begin{equation}\label{mpnehari}
d\ =\ \min_{v\in\mathcal{N}}\ J(v)\ .
\end{equation}
Using \eqref{mpnehari} and the definition of $\mathcal{N}$ we obtain
$$d\ =\ \min_{v\in \mathcal{N}}\ J(v)\ =\ \min_{v\in \mathcal{N}}\ \left(\frac{\|v\|^2}{2}-\int_\Omega\vvv\right)\
=\ \min_{v\in \mathcal{N}}\ \frac{\|v\|^2}{6}$$
which proves (i) since $\mathcal{N}$ separates $\mathcal{N}_+$ and $\mathcal{N}_-$.\par
If $v\in \mathcal{N}_+$, then $-\int_\Omega\vvv>-\|v\|^2/3$. If $J(v)<d$, then $\|v\|^2-2\int_\Omega\vvv<2d$.
By combining these two inequalities we obtain (ii).\par
Finally, recalling the definitions of $\mathcal{N}_{\pm}$, (iii) follows directly from (i).
\end{proof}

A further functional needed in the sequel is given by
\begin{equation}\label{I}
I(v)=\int_\Omega\vvv\ .
\end{equation}
We provide a different characterization of the mountain pass level.

\begin{theorem}\label{mp}
The mountain pass level $d$ for $J$ is also determined by
\begin{equation}\label{mplevel}
d=\min_{v\in B}\ \frac{\|v\|^6}{54}\ .
\end{equation}
Moreover, $d$ can be lower bounded in terms of the best constant for the (compact) embedding $W^{2,2}_0(\Omega)\subset W^{1,4}_0(\Omega)$,
namely
$$d\ge\frac{8}{27}\min_{W^{2,2}_0(\Omega)}\dfrac{(\int_\Omega|\Delta v|^2)^2}{\int_\Omega|\nabla v|^4}\ .$$
\end{theorem}
\begin{proof}
For all $v\in W^{2,2}_0(\Omega)$ consider the map $f_v:[0,+\infty)\to\mathbb{R}$ defined by
$$f_v(s)=J(sv)=\frac{s^2}{2}\int_\Omega|\Delta v|^2-s^3\int_\Omega\vvv\ .$$
If $I(v)\le0$, the map $s\mapsto f_v(s)$ is strictly increasing and strictly convex, attaining its global minimum at $s=0$;
in this case, $f_v$ has no critical points apart from $s=0$. So, the mountain pass level is achieved for some function $v$ satisfying
$I(v)>0$. For any $v\in B$, see \eqref{B}, we have
$$f_v(s)=\frac{\|v\|^2}{2}s^2-s^3\ .$$
It is straightforward to verify that the map $s\mapsto f_v(s)$ is initially increasing and then strictly decreasing.
It attains the global maximum for $s=\frac{\|v\|^2}{3}$ and
$$\max_{s\ge 0}f_v(s)=\frac{\|v\|^6}{54}\ .$$
Hence,
$$\max_{s\ge 0}J(sv)=\frac{\|v\|^6}{54}\qquad\forall v\in B\ .$$
By the minimax characterization of the mountain pass level we see that \eqref{mplevel} holds.\par
Next, note that integrating by parts we obtain
$$I(v)=\frac{1}{2}\int_\Omega v_x(v_y^2)_x=-\frac{1}{2}\int_\Omega v_{xx}\, v_y^2=-\frac{1}{2}\int_\Omega \Delta v\, v_y^2+
\frac{1}{6}\int_\Omega(v_y^3)_y\ ,$$
$$I(v)=\frac{1}{2}\int_\Omega v_y(v_x^2)_y=-\frac{1}{2}\int_\Omega v_{yy}\, v_x^2=-\frac{1}{2}\int_\Omega\Delta v\, v_x^2+
\frac{1}{6}\int_\Omega(v_x^3)_x\ ,$$
$$\qquad\forall v\in W^{2,2}_0(\Omega)\ .$$
Adding these expressions and invoking the divergence theorem leads to
\begin{equation}\label{utile}
I(v)=-\frac{1}{4}\int_\Omega\Delta v\, |\nabla v|^2+\frac{1}{12}\int_\Omega[(v_x^3)_x+(v_y^3)_y]=
-\frac{1}{4}\int_\Omega\Delta v\, |\nabla v|^2$$
$$\qquad\forall v\in W^{2,2}_0(\Omega)\ .
\end{equation}
Therefore, by H\"older inequality,
\begin{equation}\label{stima}
I(v)\le\frac{1}{4}\left(\int_\Omega|\Delta v|^2\right)^{1/2}\left(\int_\Omega|\nabla v|^4\right)^{1/2}
\qquad\forall v\in W^{2,2}_0(\Omega)
\end{equation}
and, according to \eqref{mplevel}, we infer
$$d=\frac{1}{54}\min_X\frac{(\int_\Omega|\Delta v|^2)^3}{I(v)^2}\ge\frac{8}{27}\min_{W^{2,2}_0(\Omega)}
\frac{(\int_\Omega|\Delta v|^2)^2}{\int_\Omega|\nabla v|^4}$$
where $X:=\{v\in W^{2,2}_0(\Omega);\, I(v)>0\}$.\end{proof}

In Figure \ref{picture} we sketch a geometric representation of the Nehari manifold $\mathcal{N}$ which summarizes the results obtained in the
present section.

\begin{figure}[ht]
\begin{center}
{\includegraphics[height=30mm, width=110mm]{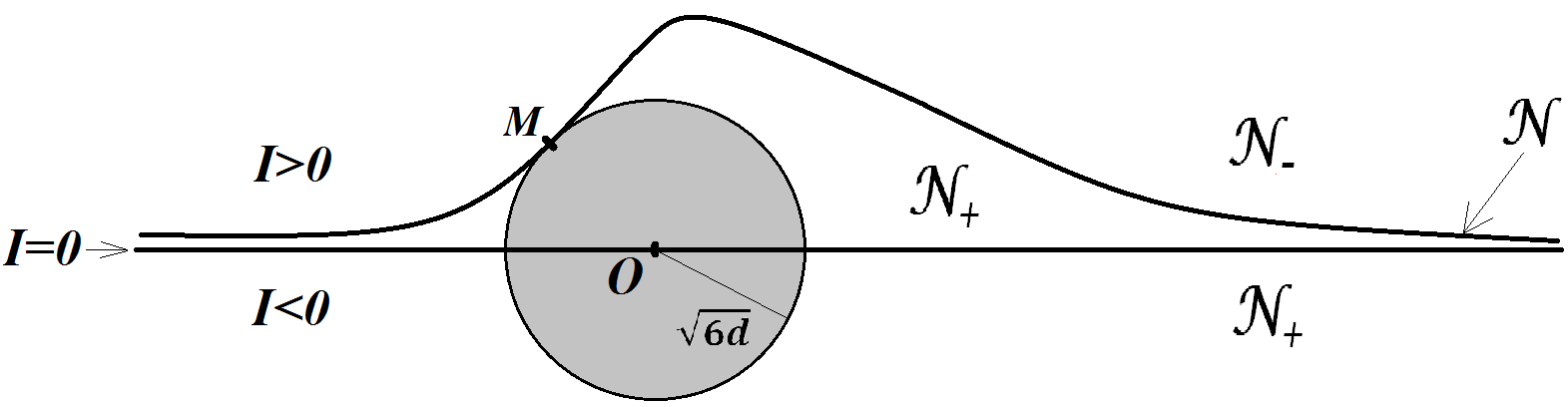}}
\caption{The phase space $W^{2,2}_0(\Omega)$ with:
$\mathcal{N}$ = Nehari manifold,\, $M$ = mountain pass point, and $I$ given by (\ref{I}).}\label{picture}
\end{center}
\end{figure}

\section{The parabolic problem with source}\label{parabolicwithsource}

This section is devoted to the study of the evolution problem
\begin{equation}\label{nonlinpar}
u_t + \Delta^2 u = \det(D^2 u) + \lambda f\qquad\mbox{in }\Omega\times(0,T)
\end{equation}
for some $T>0$. We consider both the sets of boundary conditions $u|_{\partial \Omega} = u_\nu|_{\partial \Omega}=0$ (Dirichlet)
and $u|_{\partial \Omega} = \Delta u|_{\partial \Omega}=0$ (Navier). Here and in the sequel we will be always
considering weak solutions.

We start by proving a result concerning an associated linear problem.

\begin{theorem}\label{existlin}
Let $0<T\le\infty$ and let $f \in L^2(0,T;L^2(\Omega))$. The Dirichlet problem for the linear fourth order parabolic equation
\begin{equation}\label{linpar}
u_t + \Delta^2 u = f\qquad\mbox{in }\Omega\times(0,T),
\end{equation}
with initial datum $u_0 \in W_0^{2,2}(\Omega)$ admits a unique weak solution in the space
$$
C([0,T);W_0^{2,2}(\Omega)) \cap L^2(0,T;W^{4,2}(\Omega)) \cap W^{1,2}(0,T;L^2(\Omega)).
$$
The corresponding Navier problem with initial datum
$u_0 \in W^{2,2}(\Omega) \cap W_0^{1,2}(\Omega)$ admits a unique weak solution in the space
$$
C([0,T);W^{2,2}(\Omega) \cap W_0^{1,2}(\Omega)) \cap L^2(0,T;W^{4,2}(\Omega)) \cap W^{1,2}(0,T;L^2(\Omega)).
$$
Furthermore, both cases admit the estimate
$$
\sup_{0 \le t < T} \|\Delta u\|_2^2 + \int_0^T \|\Delta^2 u\|_ 2^2 + \int_0^T \|u_t\|_2^2 \le
C \left( \|\Delta u_0\|_2^2 + \int_0^T \|f\|_2^2 \right)\ .
$$
\end{theorem}

\begin{proof}
\textsc{Step 1. Existence via Galerkin method.}
We will focus herein on Dirichlet boundary conditions; the proof for the Navier problem follows with obvious modifications.
Let $u_0\in W^{2,2}_0(\Omega)$ and consider the following linear problem
\neweq{linear}
\left\{\begin{array}{ll}
u_t+\Delta^2 u=f & \text{in } \Omega\times(0,T) \\
u=u_\nu=0 & \text{on }\partial\Omega\times(0,T) \\
u(x,0)=u_0(x) & \text{in }\Omega\, .
\end{array}\right.\endeq

Let $\{w_k\}_{k\ge 1}\subset W^{2,2}_0(\Omega)$ be an orthogonal complete system of eigenfunctions of $\Delta^2$ under Dirichlet boundary conditions
normalized by $\|w_k\|_2=1$. Denote by $\{\lambda_k\}$ the unbounded sequence of corresponding eigenvalues and by
$$W_k:={\rm span } \{w_1,\dots,w_k\}\qquad\forall k\ge 1.$$
Denote by $(\cdot,\cdot)_2$ and $(\cdot,\cdot)$ the scalar products in $L^2(\Omega)$ and $W^{2,2}_0(\Omega)$. For any $k\ge 1$ let
$$u_0^k:=\sum_{i=1}^k (u_0,w_i)_2 w_i =\sum_{i=1}^k\lambda_i^{-1}(u_0,w_i)\, w_i$$
so that $u_0^k\to u_0$ in $W^{2,2}_0(\Omega)$ as $k\to +\infty$. For any $k\ge 1$ we seek a solution $u_k\in W^{1,2}(0,T;W_k)$ of the variational problem
\begin{equation} \label{eq:P-W-k}
\left\{\begin{array}{lll}
(u'(t),v)_2+(u(t),v)=(f(t),v)_2 \\
\qquad \text{for any } v\in W_k\quad\mbox{for a.e. }t\in(0,T)\\
u(0)=u_0^k\, .
\end{array}\right.
\end{equation}
We seek solutions in the form
$$u_k(t)=\sum_{i=1}^k g_i^k(t) w_i$$
so that for any $1\le i\le k$ the function $g_i^k$ solves the Cauchy problem
\begin{equation} \label{eq:Cauchy-fin-dim}
\left\{\begin{array}{ll}
(g_i^k(t))'+\lambda_i g_i^k(t)=(f(t),w_i)_{2} \\
g_i^k(0)=(u_0^k,w_i)_{2}\, .
\end{array}\right.
\end{equation}
The linear ordinary differential equation \eqref{eq:Cauchy-fin-dim} admits a unique solution $g_i^k$ such that
$g_i^k\in W^{1,2}(0,T)$, and hence also \eqref{eq:P-W-k} admits
$u_k\in W^{1,2}(0,T;W_k)$ as a unique solution.

Note that
$$\Delta^2u_k(t)=\sum_{i=1}^k g_i^k(t)\lambda_i w_i\in W_k\qquad\mbox{for a.e. }t\in(0,T)$$
so that by testing equation \eqref{eq:P-W-k} with $v=\Delta^2u_k(t)$ we obtain that for a.e.\ $t\in(0,T)$:
$$
\frac 12 \frac{d}{dt} \|u_k(t)\|^2+\frac 12 \|u_k(t)\|_{W^{4,2}(\Omega)}^2=(f(t),\Delta^2 u_k(t))_2.
$$
After integration over $(0,t)$ we obtain
$$
\|u_k(t)\|^2-\|u_0^k\|^2+\|u_k\|_{L^2(0,t;W^{4,2}(\Omega))}^2\! \le\!
\int_0^T\!\left(C\|f(s)\|_2^2+\frac{1}{2}\|u_k(s)\|_{W^{4,2}(\Omega)}^2\right)ds
$$
and therefore
$$
\|u_k\|_{L^\infty(0,T;W^{2,2}_0(\Omega))}^2+\frac12 \|u_k\|_{L^2(0,T;W^{4,2}(\Omega))}^2 \le
\|u_0^k\|^2+C\|f\|_{L^2(0,T;L^2(\Omega))}^2.
$$
Since the sequence $\{u_0^k\}$ is bounded in $W^{2,2}_0(\Omega)$, we infer that
$$
\{u_k\}\quad\mbox{is bounded in}\quad L^\infty(0,T;W^{2,2}_0(\Omega))\cap L^2(0,T;W^{4,2}(\Omega)).
$$
Whence, we may extract a subsequence, still denoted by $\{u_k\}$ such that
$$
u_k\rightharpoonup^*u\mbox{ in }L^\infty(0,T;W^{2,2}_0(\Omega))\qquad\mbox{and}\qquad u_k\rightharpoonup u\mbox{ in }L^2(0,T;W^{4,2}(\Omega))\ .
$$
Moreover, since $u_k'=-\Delta^2u_k+f$ in the weak sense, we also have that $u_k'\in L^2(0,T;L^2(\Omega))$ and that
$$u_k'\rightharpoonup u'\mbox{ in }L^2(0,T;L^2(\Omega)).$$
Hence, by letting $k\to\infty$ in \eq{eq:P-W-k}, we see that
$$u\in L^\infty(0,T;W^{2,2}_0(\Omega))\cap L^2(0,T;W^{4,2}(\Omega))\cap W^{1,2}(0,T;L^2(\Omega))$$
solves the problem
\begin{equation} \label{weaklinear}
\left\{\begin{array}{lll}
(u'(t),v)_2+(u(t),v)=(f(t),v)_2 \\
\qquad \text{for any } v\in W^{2,2}_0(\Omega)\quad\mbox{for a.e. }t\in(0,T)\\
u(0)=u_0\, .
\end{array}\right.
\end{equation}
By interpolation between $L^2(0,T;W^{4,2}(\Omega))$ and $W^{1,2}(0,T;L^2(\Omega))$ we obtain $u\in C([0,T);W^{2,2}_0(\Omega))$.\par

\textsc{Step 2. Estimates.}
The existence result justifies the following calculations performed in order to obtain the desired estimate.
We multiply equation~\eqref{linpar} by $\Delta^2 u$ and integrate by parts over $\Omega$ the result to find
$$
\frac{1}{2} \frac{d}{dt} \|\Delta u\|_2^2 + \| \Delta^2 u \|_2^2 = ( \Delta^2 u, f )_2 \le \frac{\epsilon}{2} \|\Delta^2 u\|_2^2 +
\frac{1}{2 \epsilon} \|f\|_2^2\quad\mbox{for a.e. }t\in(0,T)
$$
for any $\epsilon>0$. Upon integration in time we obtain
$$
\sup_{0 \le t < T} \|\Delta u\|_2^2 + \int_0^T \|\Delta^2 u\|_ 2^2 \le
C \left( \|\Delta u_0\|_2^2 + \int_0^T \|f\|_2^2 \right).
$$
To conclude multiply equation \eqref{linpar} by an arbitrary function $v \in L^2(\Omega)$ to get
$$
( v, u_t )_2 + ( v, \Delta^2 u )_2 = ( v, f )_2\quad\mbox{for a.e. }t\in(0,T).
$$
This equality implies the inequality
$$
( v, u_t )_2 \le \|f\|_2 \|v\|_2 + \|\Delta^2 v\|_2\|v\|_2.
$$
Now taking the supremum over all $v \in L^2(\Omega)$ such that $\|v\|_2 = 1$ and the fact
$$
\sup_v ( v, u_t )_2 = \|u_t\|_2,
$$
we find
$$
\int_0^T \|u_t\|_2^2 \le C \left( \int_0^T \|\Delta^2 u\|_2^2 + \int_0^T \|f\|_2^2 \right),
$$
and the desired inequality follows immediately.

Finally, uniqueness follows by a standard contradiction argument.
\end{proof}

We now state the main result of this section.

\begin{theorem}\label{existence}
The problem
\neweq{target}
\left\{\begin{array}{ll}
u_t + \Delta^2 u = \det(D^2 u) + \lambda f\qquad & \mbox{in }\Omega\times(0,T)\\
u(x,0)=u_0(x)\qquad & \mbox{in }\Omega\\
u(x,t)=u_\nu(x,t)=0\qquad & \mbox{on }\partial\Omega\times(0,T)
\end{array}\right.
\endeq
admits a unique solution in
$$
\mathcal{X}_T:=C([0,T);W_0^{2,2}(\Omega)) \cap L^2(0,T;W^{4,2}(\Omega)) \cap W^{1,2}(0,T;L^2(\Omega)),
$$
provided one of the following set of conditions holds\par
(i) $u_0\in W_0^{2,2}(\Omega)$, $f\in L^2(0,T;L^2(\Omega))$, $\lambda\in \mathbb{R}$, and $T>0$ is sufficiently small;\par
(ii) $T\in(0,\infty)$, $f\in L^2(0,T;L^2(\Omega))$, and $\|u_0\|$ and $|\lambda|$ are sufficiently small.\par\noindent
Moreover, if $[0,T^*)$ denotes the maximal interval of continuation of $u$ and if $T^*<\infty$ then $\|u(t)\|\to \infty$ as $t\to T^*$.\par
An identical result holds for the Navier problem but this time the solution belongs to the space
$$
\mathcal{Y}_T:=C([0,T);W^{2,2}(\Omega) \cap W_0^{1,2}(\Omega)) \cap L^2(0,T;W^{4,2}(\Omega)) \cap W^{1,2}(0,T;L^2(\Omega)),
$$
assuming that the initial condition $u_0 \in W^{2,2}(\Omega) \cap W_0^{1,2}(\Omega)$.
\end{theorem}
\begin{proof} For all $u\in W^{4,2}(\Omega)$ we have
\begin{eqnarray*}
\|\det (D^2 u)\|_2^2 &=& \int_\Omega |\det (D^2 u)|^2 \le C \int_\Omega |D^2 u|^4\le C \|D^2 u\|_\infty^2 \int_\Omega |D^2 u|^2\\
\ &\le & C \|\Delta u\|_\infty^2 \, \|\Delta u\|^2_2 \le C \|\Delta^2 u\|_2^2 \,\, \|\Delta u\|^2_2,
\end{eqnarray*}
where the determinant of the Hessian matrix is estimated with the Euclidean norm of this matrix squared in the first inequality,
a H\"older inequality in the second inequality, the estimation of homogeneous Sobolev norms with the corresponding norms of Laplacians
in the third, and the Sobolev embedding $W^{4,2}(\Omega) \hookrightarrow W^{2,\infty}(\Omega)$ in the fourth.
Hence, if $u \in C([0,T);W^{2,2}(\Omega)) \cap L^2(0,T;W^{4,2}(\Omega))$, we may directly estimate
\begin{eqnarray*}
\|\det (D^2 u)\|_{L^2(0,T;L^2(\Omega))}^2 &=& \int_0^T \|\det (D^2 u)\|_2^2 \le C \int_0^T \|\Delta^2 u\|_2^2 \,\, \|\Delta u\|^2_2\\
\ &\le & C \sup_{0 \le t < T} \|\Delta u\|^2_2 \int_0^T \|\Delta^2 u\|_2^2 < \infty
\end{eqnarray*}
which proves that
\begin{equation}\label{inclusion}
u\! \in\! C([0,T);W^{2,2}(\Omega))\! \cap\! L^2(0,T;W^{4,2}(\Omega))\Longrightarrow\det (D^2 u)\! \in\! L^2(0,T;L^{2}(\Omega)).
\end{equation}

In what follows we focus on the Dirichlet case since the proof for the Navier one follows similarly. We introduce the initial-Dirichlet linear problems
\begin{equation}\label{initDir}
\left\{
\begin{array}{ll}
(u_1)_t + \Delta^2 u_1 = \det (D^2 v_1) + \lambda f\ ,\quad u_1(x,0)=u_0(x)\ ,\\
(u_2)_t + \Delta^2 u_2 = \det (D^2 v_2) + \lambda f\ ,\quad u_2(x,0)=u_0(x)\ ,
\end{array}
\right.
\end{equation}
where $v_1,v_2 \in {\mathcal{X}_T}$. Theorem~\ref{existlin} and~\eqref{inclusion} show that $u_1,u_2 \in {\mathcal{X}_T}$.
Subtracting the equations in \eqref{initDir} we get
$$
(u_1 - u_2)_t + \Delta^2 (u_1- u_2) = \det (D^2 v_1) - \det (D^2 v_2)\ ,\quad(u_1-u_2)(x,0)=0\ ,
$$
and upon multiplying by $\Delta^2 (u_1 - u_2)$ and integrating we find
$$
( \Delta^2 (u_1 - u_2) , (u_1 - u_2)_t )_2 + ( \Delta^2 (u_1 - u_2) , \Delta^2 (u_1- u_2) )_2 = $$
$$ ( \Delta^2 (u_1 - u_2) , \det (D^2 v_1) - \det (D^2 v_2) )_2.
$$
This leads to the inequalities
$$
\frac{1}{2} \frac{d}{dt} \|\Delta (u_1 - u_2)\|_2^2 + \|\Delta^2 (u_1 - u_2)\|_2^2 \le $$
$$
\frac{1}{2} \|\Delta^2 (u_1 - u_2)\|_2^2+ \frac{1}{2} \|\det (D^2 v_1) - \det (D^2 v_2)\|_2^2
$$
and, in turn,
\begin{equation}\label{ineqexist0}
\frac{d}{dt} \|\Delta (u_1 - u_2)\|_2^2 + \|\Delta^2 (u_1 - u_2)\|_2^2 \le \|\det (D^2 v_1) - \det (D^2 v_2)\|_2^2.
\end{equation}

We split the remaining part of the proof into three steps.

\textsc{Step 1. Existence for arbitrary temporal lapses.}

We start focussing on the case $T < \infty$ and estimating the term containing the determinants
\neweq{primera}
\|\!\det (D^2 v_1)\! -\! \det (D^2 v_2)\|_2^2\! \le\! C\! \int_\Omega\! |D^2(v_1-v_2)|^2 (|D^2 v_1|\!+\!|D^2 v_2|)^2\!\le
\endeq
$$
C(\|\Delta v_1\|_\infty^2+\|\Delta v_2\|^2_\infty) \|\Delta(v_1 - v_2)\|_2^2 \le
C(\|\Delta^2 v_1\|_2^2+\|\Delta^2 v_2\|^2_2) \|\Delta(v_1 - v_2)\|_2^2,
$$
to infer from \eq{ineqexist0}
$$
\frac{d}{dt} \|\Delta (u_1 - u_2)\|_2^2 + \|\Delta^2 (u_1 - u_2)\|_2^2 \le
C(\|\Delta^2 v_1\|_2^2+\|\Delta^2 v_2\|^2_2) \|\Delta(v_1 - v_2)\|_2^2.
$$
Integrating with respect to time we obtain
\neweq{segunda}
\sup_{0 \le t < T} \|\Delta (u_1 - u_2)\|_2^2 + \int_0^T \|\Delta^2 (u_1 - u_2)\|_2^2 \le
\endeq
$$
C \sup_{0 \le t < T} \|\Delta(v_1 - v_2)\|_2^2 \int_0^T (\|\Delta^2 v_1\|_2^2+\|\Delta^2 v_2\|^2_2).
$$
Now consider a function $w \in L^2(\Omega)$ and the scalar product
$$
( w , (u_1 - u_2)_t )_2 + ( w , \Delta^2 (u_1- u_2) )_2 =( w , \det (D^2 v_1) - \det (D^2 v_2) )_2.
$$
We have the estimate
$$
( w , (u_1 - u_2)_t )_2 \le
$$
$$
\|w\|_2 \|\Delta^2 (u_1 - u_2)\|_2 + \|w\|_2 \|\det (D^2 v_1) - \det (D^2 v_2)\|_2,
$$
and taking the supremum of all $w \in L^2(\Omega)$ such that $\|w\|_2 =1$ we get
$$
\sup_w ( w , (u_1 - u_2)_t )_2 \le \|\Delta^2 (u_1 - u_2)\|_2 + \|\det (D^2 v_1) - \det (D^2 v_2)\|_2.
$$
Therefore, from \eq{primera} we infer that
$$
\|(u_1 - u_2)_t\|_2^2 \le
$$
$$
C \left[ \|\Delta^2 (u_1 - u_2)\|_2^2 + (\|\Delta^2 v_1\|_2^2+\|\Delta^2 v_2\|^2_2) \|\Delta(v_1 - v_2)\|_2^2 \right],
$$
and consequently, by using \eq{segunda},
\neweq{tercera}
\sup_{0 \le t < T} \|\Delta (u_1 - u_2)\|_2^2 + \int_0^T \|\Delta^2 (u_1 - u_2)\|_2^2 + \int_0^T \|(u_1 - u_2)_t\|_2^2 \le
\endeq
$$
C \sup_{0 \le t < T} \|\Delta(v_1 - v_2)\|_2^2 \int_0^T (\|\Delta^2 v_1\|_2^2+\|\Delta^2 v_2\|^2_2).
$$
On the space $\mathcal{X}_T$ we define the norm
$$
\|u\|_{\mathcal{X}_T}^2 := \sup_{0 \le t < T} \|\Delta u\|_2^2 + \int_0^T \|\Delta^2 u\|_2^2 + \int_0^T \|u_t\|_2^2,
$$
so that \eq{tercera} reads
\begin{equation}\label{ineqexist}
\|u_1 - u_2\|_{\mathcal{X}_T} \le C \left[ \int_0^T (\|\Delta^2 v_1\|_2^2+\|\Delta^2 v_2\|^2_2) \right]^{1/2} \|v_1 - v_2\|_{\mathcal{X}_T}.
\end{equation}

Now consider the unique solution $u_\ell$ (see Theorem \ref{existlin}) to the linear problem
$$
(u_\ell)_t + \Delta^2 u_\ell = \lambda f,
$$
with the same boundary and initial conditions as~\eqref{initDir}. Then define the ball
\neweq{Brho}
B_\rho = \{u \in {\mathcal{X}_T} : \|u-u_\ell\|_{{\mathcal{X}_T}} \le \rho\}.
\endeq
Using estimate~\eqref{ineqexist} we find
\neweq{cuarta}
\|u_i-u_\ell\|_{{\mathcal{X}_T}} \le  C \left( \int_0^T \|\Delta^2 v_i\|_2^2 \right)^{1/2} \|v_i\|_{\mathcal{X}_T} \le C \|v_i\|_{\mathcal{X}_T}^2,
\endeq
for $i=1,2$. We use the triangle inequality
\neweq{triangle}
\|v_i\|_{\mathcal{X}_T} \le \|v_i - u_\ell\|_{\mathcal{X}_T} + \|u_\ell\|_{\mathcal{X}_T}
\endeq
together with (see Theorem \ref{existlin})
\neweq{triangle2}
\|u_\ell\|_{\mathcal{X}_T}^2 \le C \left( \|\Delta u_0\|_2^2 + \lambda^2 \int_0^T \|f\|_2^2 \right)=:C\, \Gamma(\rho,u_0,\lambda,f).
\endeq
to infer from \eq{cuarta}-\eq{triangle}-\eq{triangle2} that
$$
\|u_i - u_\ell\|_{\mathcal{X}_T} \le C \left( \rho^2 + \|\Delta u_0\|_2^2 + \lambda^2 \int_0^T \|f\|_2^2 \right),
$$
and thus
$$
\|u_i - u_\ell\|_{\mathcal{X}_T} \le \rho,
$$
for small enough $\rho$, $|\lambda|$ and $\|\Delta u_0\|_2$.

By using \eq{cuarta}-\eq{triangle}-\eq{triangle2} and reasoning as before we can transform~\eqref{ineqexist} into
$$
\|u_1 - u_2\|_{\mathcal{X}_T} \le C\, \Gamma(\rho,u_0,\lambda,f)^{1/2} \|v_1 - v_2\|_{\mathcal{X}_T}.
$$
Again, for $\rho$, $|\lambda|$ and $\|\Delta u_0\|_2$ small enough we have
$$
\|u_1 - u_2\|_{\mathcal{X}_T} \le \frac{1}{2} \|v_1 - v_2\|_{\mathcal{X}_T}.
$$
The existence of a unique solution follows from the application of Banach fixed point theorem to the map
\begin{eqnarray} \nonumber
\mathcal{A}: B_\rho &\rightarrow& B_\rho \\ \nonumber
v_i &\mapsto& u_i,
\end{eqnarray}
for $i=1,2$. The case $T=\infty$ follows similarly since $\Gamma(\rho,u_0,\lambda,f)$ does not depend on how large is $T$.

\textsc{Step 2. Local existence in time.}

By the Gagliardo-Nirenberg inequality \cite{gagliardo,nirenberg},
$$\|\Delta v_i\|_\infty \le C \|\Delta v_i\|_2^{1/4} \|\nabla \Delta v_i\|_3^{3/4},\qquad(i=1,2),$$
we may go back to~\eqref{ineqexist0} and we improve \eq{primera} with
$$
\|\det (D^2 v_1) - \det (D^2 v_2)\|_2^2 \le C(\|\Delta v_1\|_\infty^2+\|\Delta v_2\|^2_\infty) \|v_1 - v_2\|^2 \le
$$
$$
C(\|\Delta v_1\|_2^{1/2} \|\nabla \Delta v_1\|_3^{3/2} + \|\Delta v_2\|_2^{1/2} \|\nabla \Delta v_2\|_3^{3/2})\|v_1 - v_2\|^2.
$$
This, together with a Sobolev embedding, leads to
$$
\frac{d}{dt} \|u_1 - u_2\|^2 + \|\Delta^2 (u_1 - u_2)\|_2^2 \le
$$
$$
C(\|\Delta v_1\|_2^{1/2} \|\Delta^2 v_1\|_2^{3/2} + \|\Delta v_2\|_2^{1/2} \|\Delta^2 v_2\|_2^{3/2})\|v_1 - v_2\|^2.
$$
An integration with respect to time then yields
$$
\sup_{0 \le t < T} \|u_1 - u_2\|^2 + \int_0^T \|\Delta^2 (u_1 - u_2)\|_2^2 \le C\sup _{0 \le t < T} \|v_1 - v_2\|^2\times
$$
$$
C \left(\sup _{0 \le t < T}\|v_1\|^{1/2} \int_0^T \|\Delta^2 v_1\|_2^{3/2} +
\sup _{0 \le t < T} \|v_2\|^{1/2} \int_0^T \|\Delta^2 v_2\|_2^{3/2} \right).
$$
We proceed making use of H\"older inequality to find
$$
\sup _{0 \le t < T} \|u_1 - u_2\|^2 + \int_0^T \|\Delta^2 (u_1 - u_2)\|_2^2 \le C \, T^{1/4} \sup _{0 \le t < T} \|v_1 - v_2\|^2\times
$$
$$
\left[\sup _{0 \le t < T}\|v_1\|^{1/2} \left( \int_0^T \|\Delta^2 v_{1}\|_2^{2} \right)^{3/4} +
\sup _{0 \le t < T} \|v_2\|^{1/2} \left( \int_0^T \|\Delta^2 v_{2}\|_2^{2} \right)^{3/4} \right].
$$

Combining the estimates above with the arguments in Step $1$ yields
$$
\|u_1 - u_2\|_{\mathcal{X}_T} \le C \, T^{1/4} \|v_1 - v_2\|_{\mathcal{X}_T}\times
$$
$$
\left[ \sup _{0 \le t < T}\|v_1\|^{1/2} \left( \int_0^T \|\Delta^2 v_{1}\|_2^{2} \right)^{3/4} +
\sup _{0 \le t < T} \|v_2\|^{1/2} \left( \int_0^T \|\Delta^2 v_{2}\|_2^{2} \right)^{3/4} \right]^{1/2}.
$$
Consider again the ball $B_\rho$ defined in \eq{Brho}. In this case we have
$$
\|u_i - u_\ell\|_{\mathcal{X}_T} \le C \, T^{1/4}
\sup _{0 \le t < T}\|v_i\|^{1/4} \left( \int_0^T \|\Delta^2 v_{i}\|_2^{2} \right)^{3/8}\|v_i\|_{\mathcal{X}_T}
$$
$$
\le C \, T^{1/4} \|v_i\|_{\mathcal{X}_T}^2,
$$
for $i=1,2$. Arguing as in Step $1$ of the present proof we get
$$
\|u_i - u_\ell\|_{\mathcal{X}_T} \le C \, T^{1/4} \Gamma(\rho,u_0,\lambda,f),
$$
and thus
$$
\|u_i - u_\ell\|_{\mathcal{X}_T} \le \rho,
$$
for small enough $T$. Additionally we have
$$
\|u_1 - u_2\|_{\mathcal{X}_T} \le C \, T^{1/4}\Gamma(\rho,u_0,\lambda,f)^{1/2} \|v_1 - v_2\|_{\mathcal{X}_T}.
$$
Again, for $T$ small enough we find
$$
\|u_1 - u_2\|_{\mathcal{X}_T} \le \frac{1}{2} \|v_1 - v_2\|_{\mathcal{X}_T}.
$$
The existence of a unique solution to \eq{target} follows from the application of Banach fixed point theorem to the map
\begin{eqnarray} \nonumber
\mathcal{A}: B_\sigma &\rightarrow& B_\sigma \\ \nonumber
v_i &\mapsto& u_i\qquad (i=1,2).
\end{eqnarray}
We have so found $\overline{T}=\overline{T}(\lambda,\|u_0\|)$ such that \eq{target} admits a unique solution over $[0,T]$ for all $T<\overline{T}$.

\textsc{Step 3. Blow-up.}

We argue by contradiction. Assume that $[0,T^*)$, with $T^* < \infty$, is the maximal interval of continuation of the solution, and that
$\liminf_{t \to T^*} \|u(t)\|=\gamma<\infty$. Then there exists a sequence $\{t_n\}$ such that $t_n\to T^*$ and $\|u(t_n)\|<2\gamma$ for $n$ large enough.
Take $n$ sufficiently large so that $t_n+\overline{T}(\lambda,2\gamma)>T^*$, where $\overline{T}$ is defined at the end of Step 2.
Consider $u(t_n)$ as initial condition to \eq{target}. Then Step 2 tells us that the solution may be continued beyond $T^*$, contradiction.
\end{proof}

\begin{corollary}
Let $u$ be a solution as described in Theorem~\ref{existence} during the time interval $(0,T]$. Then there exists a real number $\epsilon>0$
such that the solution can be prolonged to the interval $(0,T+\epsilon]$.
\end{corollary}

\begin{proof}
This result is a consequence of Step $3$ in the proof of Theorem~\ref{existence}.
\end{proof}

It is possible to prove higher regularity of the solution if we neglect the source term.

\begin{corollary}\label{coco}
Let $u$ be a solution as described in Theorem~\ref{existence} to equation~\eqref{nonlinpar} with $\lambda=0$.
Then $u^2\in C^1(0,T;L^1(\Omega))$.
\end{corollary}

\begin{proof}
The regularity proven in Theorem~\ref{existence} for the solution $u$ to~\eqref{nonlinpar} implies that
$\det(D^2u)\in C([0,T);L^1(\Omega))$ and $\Delta^2u\in C([0,T);W^{-2,2}(\Omega))$ so that
$$u_t=-\Delta^2 u+\det(D^2u)\in C([0,T);W^{-2,2}(\Omega))$$
and, in turn, $u\in C^1(0,T;W^{-2,2}(\Omega))$. Combined with $u\in C([0,T);W^{2,2}_0(\Omega))$ this yields $uu_t\in C([0,T);L^1(\Omega))$ and, additionally,
$u^2\in C^1(0,T;L^1(\Omega))$.
\end{proof}

The following result bounds the growth of the norm of solutions.
\begin{theorem}
If $u\in \mathcal{X}_T$ solves \eqref{target} then,
\begin{eqnarray}\label{stability}
\forall \, M,\epsilon>0\quad\exists \, \tau=\tau(M,\epsilon)>0\quad: \\ \nonumber
\qquad\Big(\|u_0\|<M\, ,\quad t<\tau\Big)\ \Longrightarrow\ \|u(t)\|<M+\epsilon\ .
\end{eqnarray}
A similar statement holds for the corresponding Navier problem.
\end{theorem}

\begin{proof} We focus on the Dirichlet problem as the proof for the Navier case follows identically. We compute
$$
\frac{1}{2}\frac{d}{dt} \|\Delta u\|_2^2 = \langle \Delta u_t, \Delta u \rangle = \left( \Delta^2 u, u_t \right)_2 =
$$
$$
-\|\Delta^2 u\|_2^2 + \left( \Delta^2 u, \det(D^2 u) \right)_2 + \left( \Delta^2 u, \lambda f \right)_2 \le
$$
$$
-\|\Delta^2 u\|_2^2 + \|\Delta^2 u \|_2 \|\det(D^2 u)\|_2 + |\lambda| \, \|\Delta^2 u\|_2 \|f\|_2,
$$
by means of the application of the boundary conditions, the application of the equation and H\"older inequality. Young inequality leads to
$$
\frac{d}{dt} \|u\|^2 \le \|\det(D^2 u)\|_2^2 + \lambda^2 \|f\|_2^2;
$$
now choosing $0 < \tau < T$ and integrating in time along the interval $(0,\tau)$ we find
$$
\|u(\tau)\|^2 \le \|u_0\|^2 + \int_0^\tau \|\det(D^2 u)\|_2^2 + \lambda^2 \int_0^\tau \|f\|_2^2 <
$$
$$
M^2 + \int_0^\tau \|\det(D^2 u)\|_2^2 + \lambda^2 \int_0^\tau \|f\|_2^2.
$$
Arguing as in Step $2$ of the proof of Theorem~\ref{existence} we transform this inequality into
$$
\|u(\tau)\|^2 < M^2 + C \sup _{0 \le t < T} \|u\|^{5/2}
\left( \int_0^T \|\Delta^2 u\|_2^2 \right)^{3/4} \tau^{1/4} + \lambda^2 \int_0^\tau \|f\|_2^2.
$$
Using the concavity of the square root we conclude
$$
\|u(\tau)\| < M + C \sup _{0 \le t < T} \|u\|^{5/4}
\left( \int_0^T \|\Delta^2 u\|_2^2 \right)^{3/8} \tau^{1/8} + |\lambda| \left( \int_0^\tau \|f\|_2^2 \right)^{1/2}
$$
and the statement follows by choosing a small enough $\tau$.
\end{proof}

\section{The parabolic problem without source}\label{parabolicwithoutsource}

In this section we consider the parabolic problem
\begin{equation}\label{parabolicpde}
\left\{\begin{array}{rclll}
u_t + \Delta^2 u &=& \det(D^2 u)  &\qquad (x,t) \in \Omega \times(0,T), \\
u(x,0)&=&u_0(x) &\qquad x \in \Omega \\
u=u_\nu&=&0  & \qquad (x,t) \in \partial \Omega\times(0,T).
\end{array}
\right.
\end{equation}

\subsection{Preliminary lemmas}

We start with the following result.
\begin{lemma}\label{potentialwell}
If $u=u(t)$ solves \eqref{parabolicpde} then its energy
$$J(u(t))=\frac{1}{2}\int_\Omega|\Delta u(t)|^2-\int_\Omega\uuut$$
satisfies
$$\frac{d}{dt}J(u(t))=-\int_\Omega u_t(t)^2\le0\ .$$
\end{lemma}
\begin{proof} Two integrations by parts show that
$$\frac{d}{dt}\int_\Omega|\Delta u|^2=2\langle\Delta u_t, \Delta u\rangle=2\int_\Omega u_t\Delta^2 u\ .$$

Note that for any smooth function $v\in \mathcal{X}_T$,
$$\frac{d}{dt}\int_\Omega\vvv=\int_\Omega(v_{xt}v_yv_{xy}+v_xv_{yt}v_{xy}+v_xv_yv_{xyt})$$
and, since $v_t=0$ on $\partial\Omega$, integrating by parts we obtain
$$\int_\Omega v_{xt}v_yv_{xy}=-\int_\Omega v_t\, \Big(v_yv_{xy}\Big)_x\ ,\qquad\int_\Omega v_xv_{yt}v_{xy}
=-\int_\Omega v_t\, \Big(v_xv_{xy}\Big)_y\ ,$$
$$\int_\Omega v_xv_yv_{xyt}=\int_\Omega v_t\, \Big(v_xv_y\Big)_{xy}\ ,$$
where, in the latter, we also used the condition that $|\nabla v|=0$ on $\partial\Omega$. By collecting terms, this proves that
$$\frac{d}{dt}\int_\Omega\vvv=\int_\Omega \detv\, v_t\ .$$
By a density argument, the same holds true for the solution $u\in \mathcal{X}_T$ to \eqref{parabolicpde}. Hence,
$$\frac{d}{dt}J(u(t))=\int_\Omega\Big(\Delta^2u-\detu\Big)\, u_t=-\int_\Omega u_t^2\ ,$$
which proves the statement.\end{proof}

\begin{lemma}\label{twonehari}
Let $u_0\in W^{2,2}_0(\Omega)$ be such that $J(u_0)<d$. Then:\par
(i) if $u_0\in\mathcal{ N}_-$ the solution $u=u(t)$ to \eqref{parabolicpde} satisfies $J(u(t))<d$ and $u(t)\in \mathcal{N}_-$ for all $t\in(0,T)$;\par
(ii) if $u_0\in\mathcal{ N}_+$ the solution $u=u(t)$ to \eqref{parabolicpde} satisfies $J(u(t))<d$ and $u(t)\in\mathcal{ N}_+$ for all $t\in(0,T)$.
\end{lemma}
\begin{proof} If $J(u_0)<d$, then $J(u(t))<d$ for all $t\in(0,T)$ in view of Lemma \ref{potentialwell}. Assume moreover that
$u_0\in\mathcal{ N}_+$ and, for contradiction, that $u(t)\not\in\mathcal{ N}_+$ for some $t\in(0,T)$. Then, necessarily $u(t)\in\mathcal{ N}$
for some $t\in(0,T)$ so that, by \eqref{mpnehari}, $J(u(t))\ge d$, contradiction. We may argue similarly if $u_0\in \mathcal{N}_-$.\end{proof}
Next, we prove a kind of {\it $L^2$-Cauchy property} for global solutions with bounded energy.

\begin{lemma}\label{cauchy}
Let $u_0\in W^{2,2}_0(\Omega) $ and let $u=u(t)$ be the corresponding solution to \eqref{parabolicpde}. Then
$$
\|u(t+\delta)-u(t)\|_2^2 \le\delta\Big(J(u(t))-J(u(t+\delta))\Big)\qquad\forall\delta>0
$$
and
\begin{equation}\label{incrementi}
\left(\frac{\|u(t+\delta)\|_2-\|u(t)\|_2}{\delta}\right)^2\le\frac{J(u(t))-J(u(t+\delta))}{\delta}\ .
\end{equation}
In particular, the map $t\mapsto\|u(t)\|_2$ is differentiable and
$$\left(\frac{d}{dt}\|u(t)\|_2\right)^2\le-\frac{d}{dt}J(u(t))\ .$$
\end{lemma}
\begin{proof} By H\"older inequality, Fubini Theorem, and Lemma \ref{potentialwell}, we get
\begin{eqnarray*}
\|u(t+\delta)-u(t)\|_2^2 & = & \int_\Omega\left|\int_t^{t+\delta}u_t(\tau)\right|^2\le\delta\int_\Omega\int_t^{t+\delta}u_t(\tau)^2\\
\ & = & \delta\int_t^{t+\delta}\left(\int_\Omega u_t(\tau)^2\right)=\delta\Big(J(u(t))-J(u(t+\delta))\Big)
\end{eqnarray*}
which is the first inequality. By the triangle inequality and the just proved inequality we infer that
$$
\Big(\|u(s+\delta)\|_2-\|u(s)\|_2\Big)^2\le\|u(s+\delta)-u(s)\|_2^2 \le\delta\Big(J(u(t))-J(u(t+\delta))\Big)
$$
$$
\qquad\forall\delta>0
$$
which we may rewrite as \eqref{incrementi}. Finally, the estimate of the derivative follows by letting $\delta\to0$.\end{proof}

Also the derivative of the squared $L^2$-norm has an elegant form:

\begin{lemma}\label{ddt}
Let $u_0\in W^{2,2}_0(\Omega)$ and let $u=u(t)$ be the corresponding solution to \eqref{parabolicpde}. Then for all $t\in[0,T)$
we have
\begin{equation}\label{obtain}
\frac{1}{2}\frac{d}{dt}\|u(t)\|_2^2+\|u(t)\|^2-3\int_\Omega\uuut=0\ .
\end{equation}
\end{lemma}
\begin{proof} Multiply \eqref{parabolicpde} by $u(t)$, integrate over $\Omega$, and apply \eqref{vdetv} to obtain \eqref{obtain}.\end{proof}

Finally, we prove that the nonlinear terms goes to the ``correct'' limit for $W^{2,2}_0(\Omega)$-bounded sequences.

\begin{lemma}\label{ltozero1}
Let $\{u_k\}$ be a bounded sequence in $W^{2,2}_0(\Omega)$. Then there exists $\overline{u}\in W^{2,2}_0(\Omega)$ such that
$u_k\rightharpoonup \overline{u}$ in $W^{2,2}_0(\Omega)$ and
$$
\int_\Omega \phi \,\, \det(D^2 u_k) \to \int_\Omega \phi \,\, \det(D^2\overline{u}) \quad \forall \phi\in W^{2,2}_0(\Omega),
$$
after passing to a suitable subsequence.
\end{lemma}
\begin{proof} The first part is immediate and follows from the reflexivity of the Sobolev space $W^{2,2}_0(\Omega)$. The second part cannot be deduced
in the same way because $L^1(\Omega)$ is not reflexive and consequently the sequence $\det(D^2 u_k)$ could converge to a measure.
For all $v,w \in C^\infty_0(\Omega)$ some integrations by parts show that
\begin{equation}\label{density}
\int_{\Omega} w \,\, \text{det} \left(D^2 v \right)=
\int_{\Omega} v_{x_1} v_{x_2} w_{x_1 x_2}
-\frac 12 v_{x_2}^2 w_{x_1 x_1}-\frac 12 v_{x_1}^2 w_{x_2 x_2}.
\end{equation}
A density argument shows that the same is true for all $v,w \in W^{2,2}_0(\Omega)$.
Therefore for any $\phi \in W^{2,2}_0(\Omega)$ and any $k$ we have
$$
\int_{\Omega} \phi \,\, \text{det} \left(D^2 u_k \right)=
\int_{\Omega} (u_k)_{x_1} (u_k)_{x_2} \phi_{x_1 x_2}
-\frac 12 (u_k)_{x_2}^2 \phi_{x_1 x_1}-\frac 12 (u_k)_{x_1}^2 \phi_{x_2 x_2}.
$$
By compact embedding we know that $u_k \rightarrow u$ strongly in $W_0^{1,4}(\Omega)$
since $u_k \rightharpoonup u$ weakly in $W^{2,2}_0(\Omega)$, and thus
$$
\lim_{k \to \infty} \int_{\Omega} \phi \,\, \text{det} \left(D^2 u_k \right)=
\int_{\Omega} \overline{u}_{x_1} \overline{u}_{x_2} \phi_{x_1 x_2}
-\frac 12 \overline{u}_{x_2}^2 \phi_{x_1 x_1}-\frac 12 \overline{u}_{x_1}^2 \phi_{x_2 x_2},
$$
after passing to a suitable subsequence. Applying again \eqref{density} leads to
$$
\lim_{k \to \infty} \int_{\Omega} \phi \,\, \text{det} \left(D^2 u_k \right)=
\int_{\Omega} \phi \,\, \text{det} \left(D^2 \overline{u} \right),
$$
after passing to a suitable subsequence.
\end{proof}

\subsection{Finite time blow-up}

Our first result proves the existence of solutions to \eqref{parabolicpde} which blow up in finite time.

\begin{theorem}\label{blowup}
Let $u_0\in \mathcal{N}_-$ be such that $J(u_0)\le d$. Then the solution $u=u(t)$ to \eqref{parabolicpde} blows up in finite time, that is, there exists $T>0$
such that $\|u(t)\|\to+\infty$ as $t\nearrow T$. Moreover, the blow up also occurs in the $W^{1,4}_0(\Omega)$-norm, that is,
$\|u(t)\|_{W^{1,4}_0(\Omega)}\to+\infty$ as $t\nearrow T$.
\end{theorem}
\begin{proof}
Again, since $u_0\not\in \mathcal{N}$, we know that, by Lemma \ref{potentialwell}, we have $J(u(t))<d$ for all $t>0$. Therefore, possibly by translating
$t$, we may assume that $J(u(0))<d$ and, from now on, we rename $u_0=u(0)$. We use here a refinement of the concavity method by Levine \cite{levine}, see
also \cite{ps,tsu}. Assume for contradiction
that the solution $u=u(t)$ to \eqref{parabolicpde}  is global and define
$$M(t):=\frac{1}{2}\int_0^t\|u(s)\|_2^2$$
so that, by Theorem \ref{existence} and Corollary \ref{coco}, $M\in C^2(0,\infty)$. Then
$$M'(t)=\frac{\|u(t)\|_2^2}{2}$$
and, by \eqref{obtain},
$$M''(t)=-3J(u(t))+\frac{\|u(t)\|^2}{2}\ .$$
By the assumptions on $u_0$ and by Lemma \ref{twonehari} we know that $u(t)\in N_-$ for all $t\ge0$. In turn, by Theorem \ref{NN}, we infer
that $\|u(t)\|^2>6d$ for all $t\ge0$. Hence, recalling Lemma \ref{potentialwell} and the assumptions, we get
$$M''(t)\ge-3J(u_0)+\frac{\|u(t)\|^2}{2}>3(d-J(u_0))>0\qquad\mbox{for all }t\ge0\ .$$
This shows that
\begin{equation}\label{infinity}
\lim_{t\to\infty}M(t)=\lim_{t\to\infty}M'(t)=+\infty\ .
\end{equation}
By Lemma \ref{potentialwell} we also infer that
$$J(u(t))=J(u_0)-\int_0^t\|u_t(s)\|_2^2$$
so that
$$M''(t)=3\int_0^t\|u_t(s)\|_2^2-3J(u_0)+\frac{\|u(t)\|^2}{2}>3\int_0^t\|u_t(s)\|_2^2$$
since $\|u(t)\|^2>6d>6J(u_0)$. By multiplying the previous inequality by $M(t)>0$ and by using H\"older inequality, we get
$$\begin{array}{rcl}
M''(t)M(t)&\ge&\displaystyle \frac{3}{2}\int_0^t\|u_t(s)\|_2^2\, \int_0^t\|u(s)\|_2^2\ge\\
\displaystyle \frac32\left(\int_0^t\int_\Omega u(s)u_t(s)\right)^2&=&
\displaystyle \frac32\Big(M'(t)-M'(0)\Big)^2\ .
\end{array}$$
By \eqref{infinity} we know that there exists $\tau>0$ such that $M'(t)>7M'(0)$ for $t>\tau$ so that the latter inequality becomes
\begin{equation}\label{secondaprima}
M''(t)M(t)>\frac{54}{49}M'(t)^2\qquad\mbox{for all }t>\tau\ .
\end{equation}
This shows that the map $t\mapsto M(t)^{-5/49}$ has negative second derivative and is therefore concave on $[\tau,+\infty)$. Since $M(t)^{-5/49}\to0$ as
$t\to\infty$ in view of \eqref{infinity}, we reach a contradiction. This shows that the solution $u(t)$ is not global and, by Theorem
\ref{existence}, that there exists $T>0$ such that $\|u(t)\|\to+\infty$ as $t\nearrow T$.\par
Since by Lemma \ref{twonehari} we have that $u(t)\in\mathcal{ N}_-$ for all $t\ge0$, by \eqref{stima} we infer that
$$
\|u(t)\|^2<3\int_\Omega\uuut\le\frac{3}{4}\|u(t)\|\, \|u(t)\|_{W^{1,4}_0(\Omega)}^2\qquad\mbox{for all }t\ge0
$$
so that $\|u(t)\|<\frac{3}{4}\|u(t)\|_{W^{1,4}_0(\Omega)}^2$ and the $W^{1,4}_0(\Omega)$-norm blows up as
$t\nearrow T$.\end{proof}

Next, we state a blow up result without assuming that the initial energy $J(u_0)$ is smaller than the mountain pass
level $d$. Let $\lambda_1$ denote the least Dirichlet eigenvalue of the biharmonic operator in $\Omega$ and
assume that $u_0\in W^{2,2}_0(\Omega)$ satisfies
\begin{equation}\label{lambda1}
\lambda_1\|u_0\|_2^2>6J(u_0)\ .
\end{equation}
By Poincar\'e inequality $\|u_0\|^2\ge\lambda_1\|u_0\|_2^2$, we see that if $u_0$ satisfies
\eqref{lambda1}, then $u_0\in \mathcal{N}_-$. However, the energy $J(u_0)$ may be larger than $d$. For instance, let
$e^1$ denote an eigenfunction corresponding to $\lambda_1$ with the sign implying $\int_\Omega e^1_xe^1_ye^1_{xy}>0$.
If we take $u_0=\alpha e^1$, then \eqref{lambda1} will be satisfied for any $\alpha>\overline{\alpha}$ where $\overline{\alpha}$ is
the unique value of $\alpha>0$ such that $\overline{\alpha}e^1\in \mathcal{N}$. And, by \eqref{mpnehari},
we know that $J(\overline{\alpha}e^1)>d$. So, for $\alpha>\overline{\alpha}$ sufficiently close to $\overline{\alpha}$ we have
$J(\alpha e^1)>d$, that is, we are above the mountain pass level.\par
Assumption \eqref{lambda1} yields finite time blow-up.

\begin{theorem}\label{highenergy}
Assume that $u_0\in W^{2,2}_0(\Omega)$ satisfies \eqref{lambda1}. Then the solution $u=u(t)$ to \eqref{parabolicpde} blows up in finite time,
that is, there exists $T>0$ such that $\|u(t)\|\to+\infty$ and $\|u(t)\|_{W^{1,4}_0(\Omega)}\to+\infty$ as $t\nearrow T$.
\end{theorem}
\begin{proof} We first claim that if $u=u(t)$ is a global solution to \eqref{parabolicpde} then
\begin{equation}\label{liminf}
\liminf_{t\to\infty}\|u(t)\|<+\infty\ .
\end{equation}
For contradiction, assume that the solution $u=u(t)$ to \eqref{parabolicpde} is global and that
\begin{equation}\label{infiniteblup}
\|u(t)\|\to+\infty\qquad\mbox{as }t\to+\infty\ .
\end{equation}
In what follows, we use the same tools as in the proof of Theorem \ref{blowup}. Consider again
$$M(t):=\frac{1}{2}\int_0^t\|u(s)\|_2^2\ .$$
Then
$$M''(t)=-3J(u(t))+\frac{\|u(t)\|^2}{2}\to+\infty\qquad\mbox{as }t\to+\infty$$
because of \eqref{infiniteblup} and Lemma \ref{potentialwell} (the map $t\mapsto-3J(u(t))$ is increasing). This proves again \eqref{infinity}.\par
By Lemma \ref{potentialwell} and using \eqref{infiniteblup} we also infer that there exists $\tau>0$ such that
$$M''(t)>3\int_0^t\|u_t(s)\|_2^2\qquad\forall t>\tau\ .$$
By multiplying the previous inequality by $M(t)>0$ and by using H\"older inequality, we find
$$M''(t)M(t)\ge\frac32\Big(M'(t)-M'(0)\Big)^2\qquad\forall t>\tau$$
and that \eqref{secondaprima} holds, for a possibly larger $\tau$. The same concavity argument used in the proof of Theorem \ref{blowup}
leads to a contradiction. Hence, \eqref{infiniteblup} cannot occur and \eqref{liminf} follows.\par\smallskip
Next, by Poincar\'e inequality and Lemma \ref{potentialwell}, \eqref{obtain} yields
$$\frac{d}{dt}\|u(t)\|_2^2=-6J(u(t))+\|u(t)\|^2\ge-6J(u_0)+\lambda_1\|u(t)\|_2^2\ .$$
By putting $\psi_0(t):=-6J(u_0)+\lambda_1\|u(t)\|_2^2$, the previous inequality reads $\psi_0'(t)\ge\lambda_1\psi_0(t)$.
Since \eqref{lambda1} yields $\psi_0(0)>0$, this proves that $\psi_0(t)\to\infty$ as $t\to\infty$.
Hence, by invoking again Poincar\'e inequality, we see that also \eqref{infiniteblup} holds, a situation that
we ruled out by proving \eqref{liminf}. This contradiction shows that $T<\infty$.
The blow up of the $W^{1,4}_0(\Omega)$-norm follows as in the proof of Theorem \ref{blowup}.
\end{proof}

Let $u_0\in W^{2,2}_0(\Omega)$ and let $u=u(t)$ be the local solution to \eqref{parabolicpde}. According to Theorem \ref{existence}, the solution blows up at
some $T>0$ if
\begin{equation}\label{blupH2}
\lim_{t\to T}\ \|u(t)\|\ =\ +\infty\ .
\end{equation}
We wish to investigate if the (finite time) blow up also occurs in different ways. In particular, we wish to analyze the
following forms of blow up:
\begin{equation}\label{blupL2H2}
\lim_{t\to T}\ \|u\|_{L^2(0,t;W^{2,2}_0(\Omega))}\ =\ +\infty\ ,
\end{equation}
\begin{equation}\label{blupL2}
\lim_{t\to T}\ \|u(t)\|_2\ =\ +\infty\ ,
\end{equation}
\begin{equation}\label{blupL4W14}
\lim_{t\to T}\ \|u\|_{L^4(0,t;W^{1,4}_0(\Omega))}\ =\ +\infty\ .
\end{equation}
Clearly, \eqref{blupL2} implies \eqref{blupH2}. We show that also further implications hold true.

\begin{theorem}\label{differentblowup}
Let $u_0\in W^{2,2}_0(\Omega)$ and let $u=u(t)$ be the local solution to \eqref{parabolicpde}. Assume that \eqref{blupH2} occurs for
some finite $T>0$.
Then there exists $\tau\in(0,T)$ such that $u(t)\in\mathcal{ N}_-$ for all $t>\tau$.\par
Moreover:\par\noindent
(i) If \eqref{blupL2H2} occurs, then \eqref{blupL2} occurs.\par\noindent
(ii) If \eqref{blupL2} occurs, then \eqref{blupL4W14} occurs.\par
Finally, \eqref{blupL2} occurs if and only if
\begin{equation}\label{diverge3}
\lim_{t\to T}\int_0^t\left(\int_\Omega\Delta u(s)|\nabla u(s)|^2\right)\, =\, -\infty\ .
\end{equation}
\end{theorem}
\begin{proof}
For contradiction, assume that there exists a sequence $t_n\to T$
such that $u(t_n)\in(\mathcal{N}\cup\mathcal{ N}_+)$. Then
$$\|u(t_n)\|^2\ge3\int_\Omega\uuutn\qquad\forall n$$
which, in view of \eqref{blupH2}, implies that
$$J(u(t_n))=\frac12 \|u(t_n)\|^2-\int_\Omega\uuutn\ge\frac16 \|u(t_n)\|^2\to+\infty$$
as $n\to\infty$. This contradicts Lemma \ref{potentialwell}.
Hence, there exists $\tau\in(0,T)$ such that $u(t)\in\mathcal{ N}_-$ for all $t>\tau$.\par
Integrating \eqref{obtain} over $[0,t]$ for $0<t<T$ yields
\begin{equation}\label{integrata}
\|u(t)\|_2^2=\|u_0\|_2^2+\int_0^t \left(-2\|u(s)\|^2+6\int_\Omega\uuus\right)\ .
\end{equation}
By Lemma \ref{potentialwell} we know that $J(u(t))\le J(u_0)$, that is,
$$2\int_\Omega\uuut\ge\|u(t)\|^2-2J(u_0)\qquad\forall t\in(0,T)\ .$$
Hence, \eqref{integrata} yields
$$
\|u(t)\|_2^2\ge\|u_0\|_2^2+\int_0^t\|u(s)\|^2-6J(u_0)\, t\qquad\forall t\in(0,T)\ .
$$
Letting $t\to T$ we see that \eqref{blupL2H2} implies \eqref{blupL2}.\par
Using \eqref{stima} into \eqref{integrata} yields
\begin{equation}\label{integrata2}
\|u(t)\|_2^2\le\|u_0\|_2^2+\int_0^t \left(-2\|u(s)\|^2+\frac32 \|u(s)\|\, \|u(s)\|_{W^{1,4}_0(\Omega)}^2\right)\ .
\end{equation}
By the Young inequality $\frac32 ab\le 2a^2+\frac{9}{32}b^2$, \eqref{integrata2} becomes
$$
\|u(t)\|_2^2\le\|u_0\|_2^2+\frac{9}{32}\|u\|_{L^4(0,t;W^{1,4}_0(\Omega))}^4\ .
$$
Letting $t\to T$, this proves that if \eqref{blupL2} occurs, then also \eqref{blupL4W14} occurs.\par
Assume now that \eqref{diverge3} occurs and, using \eqref{utile}, rewrite \eqref{integrata} as
\begin{equation}\label{another}
\|u(t)\|_2^2=\|u_0\|_2^2+\int_0^t \left(-2\|u(s)\|^2-\frac32 \int_\Omega\Delta u(s)\, |\nabla u(s)|^2\right)\ .
\end{equation}
Two cases may occur. If \eqref{blupL2H2} holds, then by the just proved statement (i), \eqref{blupL2} occurs.
If \eqref{blupL2H2} does not hold, so that $\|u\|_{L^2(0,t;W^{2,2}_0(\Omega))}$ remains bounded, then \eqref{another} shows again that \eqref{blupL2}
occurs.
Therefore, in any case, if \eqref{diverge3} occurs, then \eqref{blupL2} occurs.\par
Finally, from \eqref{another} we see that
$$
\|u(t)\|_2^2\le\|u_0\|_2^2-\frac32 \int_0^t \left(\int_\Omega\Delta u(s)\, |\nabla u(s)|^2\right)
$$
which proves that \eqref{blupL2} implies \eqref{diverge3}.\end{proof}

\subsection{Global solutions}

For suitable initial data, not only the solution is global but it vanishes in infinite time.

\begin{theorem}\label{tozero}
Let $u_0\in \mathcal{N}_+$ be such that $J(u_0)\le d$. Then the solution $u=u(t)$ to \eqref{parabolicpde} is global and
$u(t)\to0$ in $W^{4,2}(\Omega)$ as $t\to+\infty$.
\end{theorem}
\begin{proof}
Since $u_0\not\in \mathcal{N}$, we know that it is not a stationary solution to \eqref{parabolicpde}, that is, it
does not solve \eqref{eqtnstat}. Hence, by Lemma \ref{potentialwell} we have $J(u(t))<d$ for all $t>0$.
By Lemma \ref{twonehari} and Theorem \ref{NN} we infer that $u(t)$ remains bounded in $W^{2,2}_0(\Omega)$ so that, by
Theorem \ref{blowup}, the solution is global.
If $\|u_t\|_2\ge c>0$ for all $t>0$, then by Lemma \ref{potentialwell} we would get $J(u(t))\to-\infty$ as $t\to\infty$ against
$u(t)\in N_+$, see again Lemma \ref{twonehari}. Hence, $u_t(t)\to0$ in $L^2(\Omega)$, on a suitable sequence.\par
Moreover, the boundedness of $\|u(t)\|$ implies that there exists $\overline{u}\in W^{2,2}_0(\Omega)$ such that $u(t)\weak\overline{u}$ in $W^{2,2}_0(\Omega)$
as $t\to\infty$ on the sequence. Note also that, by Lemma \ref{ltozero1}, for all $\phi\in W^{2,2}_0(\Omega)$ we have
\begin{equation*}
\int_\Omega \phi \,\, \detut \to \int_\Omega \phi \,\, {\rm det}(D^2\overline{u}).
\end{equation*}

Therefore, if we test \eqref{parabolicpde} with some $\phi\in W^{2,2}_0(\Omega)$, and we let $t\to\infty$ on the above
found sequence, we get
$$0=\int_\Omega u_t(t)\phi+\int_\Omega\Delta u(t)\Delta\phi-\int_\Omega\detut\phi\to\int_\Omega\Delta
\overline{u}\Delta\phi-{{\rm det}(D^2\overline{u})}\phi$$
which shows that $\overline{u}$ solves \eqref{eqtnstat}. Since the only solution to \eqref{eqtnstat} at energy level below $d$ is the trivial one,
we infer that $\overline{u}=0$. Writing \eqref{parabolicpde} as
$$
\Delta^2u(t)=-u_t(t)+\detut
$$
we see that $\Delta^2u(t)$ is uniformly bounded in $L^1(\Omega)$. Whence, by arguing as in the proof of Theorem \ref{regularity},
we first infer that $\Delta^2u(t)$ is bounded in $W^{-s,2}(\Omega)$ for all $s>1$ and then, by a bootstrap argument, that
$$
\Delta^2u(t)=-u_t(t)+\detut\to0 \hbox{ strongly in  } L^2(\Omega)
$$
so that $u(t)\to0$ in $W^{4,2}(\Omega)$ on the sequence.\par
By Lemma \ref{potentialwell}, we infer that $J(u(t))\to0$ regardless of how $t\to\infty$. Since $u(t)\in \mathcal{N}_+$
for all $t\ge0$, we also have that $J(u(t))\ge\|u(t)\|^2/6$ for all $t$. These facts enable us to conclude
that all the above convergences occur as $t\to\infty$, not only on some subsequence.\end{proof}

Theorems \ref{blowup} and \ref{highenergy} determine a wide class of initial data $u_0\in W^{2,2}_0(\Omega)$ which ensure that the solution to
\eqref{parabolicpde} blows up in finite time.
One can wonder whether the blow up might also occur in infinite time. This happens, for instance, in semilinear second order parabolic
equations at critical growth, see \cite{nst,quitsoupl}. If $T=+\infty$, we denote by
$$
\omega(u_0)= \bigcap_{t \ge 0} \overline{\{u(s)\::\: s \ge t\}}
$$
the $\omega$-limit set of $u_0\in W^{2,2}_0(\Omega)$, where the closure is taken in $W^{2,2}_0(\Omega)$.
We show here that infinite time blow up cannot occur for the fourth order parabolic equation \eqref{parabolicpde}.
In fact, since the nonlinearity $\detu$ is analytic, for any bounded trajectory the $\omega$-limit set consists of only one point
(see \cite{hale,haraux}) and we can prove the following statement.

\begin{theorem}\label{finiteblup}
Let $u_0\in W^{2,2}_0(\Omega)$ and let $u=u(t)$ be the local solution to \eqref{parabolicpde}. If $T=+\infty$ then the $\omega$-limit
set $\omega(u_0)$ consists of a solution to \eqref{eqtnstat}: this means that there exists a solution $\overline{u}$ to
\eqref{eqtnstat} such that $u(t)\to\overline{u}$ in $W^{2,2}_0(\Omega)$. This convergence is, in fact, also in $W^{4,2}(\Omega)$.
\end{theorem}
\begin{proof}
If $u=u(t)$ is a global solution to \eqref{parabolicpde}, then we know that \eqref{liminf} holds.
We claim that if
\begin{equation}\label{infiniteblup2}
C:=\liminf_{t\to\infty}\|u(t)\|<\limsup_{t\to\infty}\|u(t)\|=+\infty\ ,
\end{equation}
then $J(u(t))\ge d$ for all $t\ge0$ and $C>0$.
By Lemma \ref{potentialwell}, the map $t\mapsto J(u(t))$ admits a limit as $t\to\infty$. If this limit were smaller than $d$
(including $-\infty$), then
we would have $J(u(\overline{t}))<d$ for some $\overline{t}>0$. By \eqref{mpnehari} this implies that either
$u(\overline{t})\in\mathcal{N}_+$ or $u(\overline{t})\in \mathcal{N}_-$.
In the first situation, Theorem \ref{tozero} implies that $\|u(t)\|_{W^{4,2}(\Omega)}\to0$ as $t\to\infty$. In the second situation, Theorem \ref{blowup} implies that
$\|u(t)\|\to\infty$ in finite time. In both cases we contradict \eqref{infiniteblup2}. Hence, if \eq{infiniteblup2} holds then
\begin{equation}\label{limit}
J(u(t))\ge\Upsilon:=\lim_{t\to\infty}J(u(t))\ge d\ .
\end{equation}

If $C=0$ in \eq{infiniteblup2}, then there exists a divergent sequence $\{t_m\}$ such that $\|u(t_m)\|\to0$ so that $J(u(t_m))\to0$,
contradicting \eqref{limit}. By Lemma \ref{potentialwell} we know that
\neweq{AA}
\int_0^\infty\|u_t(t)\|_2^2 =J(u_0)-\Upsilon
\endeq
so that $u_t\in L^2(\mathbb{R}_+;L^2(\Omega))$ and
\begin{equation}\label{claim}
\liminf_{t\to\infty}\ \|u_t(t)\|_2\ =\ 0\ .
\end{equation}

We claim that also
\begin{equation}\label{claim2}
u\in L^\infty(\mathbb{R}_+;L^2(\Omega))\ .
\end{equation}
If $u\in L^\infty(\mathbb{R}_+; W^{2,2}_0(\Omega))$, then the statement follows directly from Poincar\'e inequality. So, assume that $t\mapsto\|u(t)\|$
is not bounded in $\mathbb{R}_+$ so that, by \eqref{liminf}, we know that necessarily \eqref{infiniteblup2} holds.
Let $\Lambda:=\max \{2C,8J(u_0)\}>0$ and consider the two sets
$$\Theta_-:=\{t\ge0;\, \|u(t)\|^2\le\Lambda\}\ ,\qquad\Theta_+:=\{t\ge0;\, \|u(t)\|^2>\Lambda\}\ .$$
We have $\Theta_-\cup\Theta_+=\mathbb{R}_+$ and, in view of \eqref{infiniteblup2}, both $\Theta_+\neq\emptyset$ and $\Theta_-\neq\emptyset$.
Note that for $t\in\Theta_+$ we have $\|u(t)\|^2>8J(u_0)\ge 8J(u(t))$ in view of Lemma \ref{potentialwell} so that $u(t)\in\mathcal{N}_-$
and, by \eqref{obtain}, the map $t\mapsto\|u(t)\|_2$ is strictly increasing in $\Theta_+$.
By \eq{infiniteblup2} we know that $t$ changes infinitely many times between $\Theta_+$ and $\Theta_-$.
As long as $t\in\Theta_-$, by Poincar\'e inequality we have $\lambda_1\|u(t)\|_2^2\le\|u(t)\|^2\le\Lambda$ and therefore $\|u(t)\|_2$ remains uniformly bounded.
Moreover, by the just proved monotonicity, as long as $t\in\Theta_+$ we know that $\|u(t)\|_2^2\le\|u(\overline{t})\|_2^2\le\Lambda/\lambda_1$ where
$\overline{t}$ is the first instant where $t$ exists $\Theta_+$. This proves \eqref{claim2}.\par
Next, note that if $c$ denotes positive constants which may vary from line to line, we may rewrite \eqref{obtain} as
\begin{eqnarray*}
\|u(t)\|^2 &=& 6J(u(t))+2\int_\Omega u(t)u_t(t)\le6J(u_0)+2\int_\Omega|u(t)u_t(t)|\\
&\le& c\Big(1+\|u(t)\|_2\, \|u_t(t)\|_2\Big)\le c\Big(1+\|u_t(t)\|_2\Big)
\end{eqnarray*}
where we also used Lemma \ref{potentialwell} (first inequality), H\"older inequality (second inequality), and \eqref{claim2} (third inequality).
By squaring, we obtain
\begin{equation}\label{L4}
\|u(t)\|^4 \le c_1+c_2\|u_t(t)\|_2^2\ .
\end{equation}
Put $\Gamma_t:=\{s\ge t;\, \|u(s)\|^4\ge c_1+1\}$ where $c_1$ is as in \eqref{L4} and let $|\Gamma_t|$ denote the measure of $\Gamma_t$.
Then $|\Gamma_t|\to0$ as $t\to\infty$ because of \eqref{AA} and \eqref{L4}. Take $M:=\sqrt[4]{c_1+1}$, $\epsilon=1$, and let $\tau>0$ be the number
given by \eqref{stability}. Then take $t$ sufficiently large so that $|\Lambda_t|<\tau$. By \eqref{stability}, for any such $t$ we have
$\|u(t)\|<M+1$, which proves that
\begin{equation}\label{eccola}
u\in L^\infty(\mathbb{R}_+;W^{2,2}_0(\Omega))\ .
\end{equation}

By \eqref{claim} there exists a diverging sequence $\{t_k\}$ such that $u_t(t_k)\to0$ in $L^2(\Omega)$ as $k\to\infty$. By \eqref{eccola},
up to a further subsequence, we have $u(t_k)\weak\overline{u}$ in $W^{2,2}_0(\Omega)$ for some $\overline{u}\in W^{2,2}_0(\Omega)$.
By testing \eqref{parabolicpde} with $\varphi\in C^\infty_0(\Omega)$ and letting $k\to\infty$, we see that $\overline{u}$ solves \eqref{eqtnstat}.
Moreover the analyticity of the nonlinearity implies that any subsequence converges to the same limit $\overline{u}$ solution of the stationary
problem, see \cite{hale,haraux,jendoubi,simon}. Finally, the convergence may be improved to $W^{4,2}(\Omega)$ by arguing as in Theorem \ref{tozero}.\end{proof}

Next, we prove a squeezing property which is typical of dissipative dynamical systems.
Since \eqref{parabolicpde} is indeed dissipative when dealing with global solutions, we restrict our attention to this case. Consider the
sequence of Dirichlet eigenvalues $\{\lambda_m\}$ of the biharmonic operator and denote by $\{e^m\}$ the
sequence of corresponding $W^{2,2}_0(\Omega)$-normalized orthogonal eigenfunctions. It is well-known that
$$v=\sum_{m=1}^\infty\ \left(\int_\Omega\Delta v\Delta e^m\right)\, e^m\qquad\forall v\in W^{2,2}_0(\Omega)$$
where the series converges in the $W^{2,2}_0(\Omega)$-norm. For all $k\ge2$ denote by $P_k$ the projector onto the
space $H_k$ spanned by $\{e^1,...,e^{k-1}\}$ so that
$$P_kv=\sum_{m=1}^{k-1}\ \left(\int_\Omega\Delta v\Delta e^m\right)\, e^m\qquad\forall v\in W^{2,2}_0(\Omega)\ .$$
Finally, we recall the improved Poincar\'e inequality
\begin{equation}\label{poin}
\lambda_k\|v\|_2^2\le\|v\|^2\qquad\forall v\in H_k^\perp
\end{equation}
where $H_k^\perp$ denotes the orthogonal complement of $H_k$, namely the closure of the infinite dimensional space spanned by $\{e^k,e^{k+1},...\}$.
Roughly speaking, the next result states that the asymptotic behavior of the solutions to \eqref{parabolicpde} is determined by a finite number of modes.

\begin{theorem}\label{modes}
Let $u=u(t)$ and $v=v(t)$ be the solutions to \eqref{parabolicpde} corresponding to initial data $u_0\in W^{2,2}_0(\Omega)$ and $v_0\in W^{2,2}_0(\Omega)$,
respectively. Assume that $u$ and $v$ are global solutions to \eqref{parabolicpde}. There exists $k\in\mathbb{N}$,
depending only on $\|u\|_{L^\infty(\mathbb{R}_+;W^{2,2}_0(\Omega))}$ and $\|v\|_{L^\infty(\mathbb{R}_+;W^{2,2}_0(\Omega))}$, such that
if $P_ku(t)=P_kv(t)$ for all $t\ge0$, then
$$\lim_{t\to\infty}\|u(t)-v(t)\|_{W^{s,2}(\Omega)}=0\qquad\mbox{for all }s\in[0,2)\ .$$
\end{theorem}
\begin{proof} Since $u,v\in C(\mathbb{R}_+;W^{2,2}_0(\Omega))$, by Theorem \ref{finiteblup} we know that
\begin{equation}\label{bothbounded}
u,v\in L^\infty(\mathbb{R}_+;W^{2,2}_0(\Omega))\ .
\end{equation}
We first claim that there exists $\mu>0$ such that for all $u,v\in W^{2,2}_0(\Omega)$ we have
\begin{equation}\label{picone}
\int_\Omega\Big(\detu-\detv\Big)(u-v)\le\mu\Big(\|u\|+\|v\|\Big)\|u-v\|\, \|u-v\|_\infty\ .
\end{equation}
To see this, let us rewrite
$$
\detu-\detv =
$$
$$
u_{xx}(u_{yy}-v_{yy})+v_{yy}(u_{xx}-v_{xx})+u_{xy}(v_{xy}-u_{xy})+v_{xy}(v_{xy}-u_{xy})
$$
so that, by H\"older inequality,
\begin{eqnarray*}
\|\detu-\detv\|_1 &\le& \|u_{xx}\|_2\, \|u_{yy}-v_{yy}\|_2+\|v_{yy}\|_2\, \|u_{xx}-v_{xx}\|_2\\
\ &\ & +\|u_{xy}\|_2\, \|v_{xy}-u_{xy}\|_2+\|v_{xy}\|_2\, \|v_{xy}-u_{xy}\|_2\\
\ &\le& \mu\Big(\|u\|+\|v\|\Big)\|u-v\|\ .
\end{eqnarray*}
Hence, by applying once more, H\"older inequality we obtain
\begin{eqnarray*}
\int_\Omega\Big(\detu-\detv\Big)(u-v) &\le& \|\detu-\detv\|_1\, \|u-v\|_\infty\\
\ &\le& \mu\Big(\|u\|+\|v\|\Big)\|u-v\|\, \|u-v\|_\infty\ ,
\end{eqnarray*}
which proves \eqref{picone}.\par
By subtracting the two equations relative to $u$ and $v$ we obtain
\begin{equation}\label{w}
w_t+\Delta^2w=\detu-\detv
\end{equation}
where $w(t)=u(t)-v(t)$. Multiply \eqref{w} by $w$ and integrate over $\Omega$ to obtain
$$\frac12 \frac{d}{dt}\|w(t)\|_2^2+\|w(t)\|^2=\int_\Omega\Big(\detut-\detvt\Big)w(t)\ .$$
By \eqref{picone} the latter may be estimated as
$$\frac12 \frac{d}{dt}\|w(t)\|_2^2+\|w(t)\|^2\le\mu\Big(\|u(t)\|+\|v(t)\|\Big)\|w(t)\|\, \|w(t)\|_\infty\ .$$
In turn, by the Gagliardo-Nirenberg inequality $\|w\|_\infty\le c\|w\|^{1/2}\, \|w\|_2^{1/2}$ (see \cite{gagliardo,nirenberg}) we obtain
$$\frac12 \frac{d}{dt}\|w(t)\|_2^2+\|w(t)\|^2\le
$$
$$
\mu\Big(\|u\|_{L^\infty(\mathbb{R}_+;W^{2,2}_0(\Omega))}+\|v\|_{L^\infty(\mathbb{R}_+;W^{2,2}_0(\Omega))}\Big)\|w(t)\|^{3/2}\, \|w(t)\|_2^{1/2}\ .$$
By recalling the assumption that $P_kw(t)=0$, that is $w(t)\in H_k^\perp$, and by using \eqref{poin} we then get
\begin{eqnarray*}
\frac12 \frac{d}{dt}\|w(t)\|_2^2\qquad\qquad\qquad\qquad\qquad\qquad\qquad\qquad\qquad\qquad\qquad\qquad\qquad\ \\
\le\! \bigg[\mu\!\Big[\|u\|_{L^\infty(\mathbb{R}_+;W^{2,2}_0(\Omega))}\!\!+\!\!\|v\|_{L^\infty(\mathbb{R}_+;W^{2,2}_0(\Omega))}\Big]\!
\|w(t)\|_2^{1/2}\!\!-\!\|w(t)\|^{1/2}\bigg]\|w(t)\|^{3/2}\\
\le\! \bigg[\mu\Big[\|u\|_{L^\infty(\mathbb{R}_+;W^{2,2}_0(\Omega))}\!+\!\|v\|_{L^\infty(\mathbb{R}_+;W^{2,2}_0(\Omega))}\Big]\!-\!\lambda_k^{1/4}\bigg]
\|w(t)\|_2^{1/2} \|w(t)\|^{3/2}.
\end{eqnarray*}
Take $k$ large enough so that $\lambda_k\!>\!\mu^4\Big(\|u\|_{L^\infty(\mathbb{R}_+;W^{2,2}_0(\Omega))}\!+\!\|v\|_{L^\infty(\mathbb{R}_+;W^{2,2}_0(\Omega))}\Big)^4$
and put
$$\omega_k:=\lambda_k^{1/4}-\mu\Big(\|u\|_{L^\infty(\mathbb{R}_+;W^{2,2}_0(\Omega))}+\|v\|_{L^\infty(\mathbb{R}_+;W^{2,2}_0(\Omega))}\Big)>0\ ,
$$
$$
\qquad\delta_k:=2\omega_k\lambda_k^{3/4}>0\ .$$
By \eqref{poin} we may finally rewrite the last inequality as
$$\frac{d}{dt}\|w(t)\|_2^2\le-2\omega_k\|w(t)\|_2^{1/2}\|w(t)\|^{3/2}\le-\delta_k\|w(t)\|_2^2$$
which, upon integration, gives
\begin{equation}\label{nouniform}
\|w(t)\|_2^2\le\|w(0)\|_2^2\, e^{-\delta_kt}\qquad\forall t\ge0
\end{equation}
and the statement follows for $s=0$ by letting $t\to\infty$.\par
By interpolation, we know that
$$\|u(t)-v(t)\|_{W^{s,2}(\Omega)}^2\le\|u(t)-v(t)\|_2^{2-s}\|u(t)-v(t)\|^s\qquad\mbox{for all }s\in(0,2)\, ;$$
the statement follows for all such $s$ by combining \eqref{bothbounded} and \eqref{nouniform}.\end{proof}

\section{Further results and open problems}\label{open}

\subsubsection*{Monotonicity of the $L^2$-norm.}

It is clear that the limits in \eqref{blupL2H2} and in \eqref{blupL4W14} do exist due to the fact that they involve increasing functions of $t$.
Less obvious is the existence of the limit in \eqref{blupL2}. The next result gives some monotonicity properties of the map $t\mapsto\|u(t)\|_2$
which guarantee that also the limit in \eqref{blupL2} exists.

\begin{proposition}\label{L2}
Let $\lambda_1$ denote the least eigenvalue of the biharmonic operator under Dirichlet boundary conditions in $\Omega$. Take
$u_0\in W^{2,2}_0(\Omega)$ and let $u=u(t)$ denote the corresponding local solution to \eqref{parabolicpde}.\par
(i) If \eqref{lambda1} holds, then the map $t\mapsto\|u(t)\|_2$ is strictly increasing on $[0,T)$.\par
(ii) If $u_0\in \mathcal{N}_-$ and $J(u_0)<d$, then the map $t\mapsto\|u(t)\|_2$ is strictly increasing on $[0,T)$.\par
(iii) If $u_0\in \mathcal{N}_+$ and $J(u_0)<d$, then the map $t\mapsto\|u(t)\|_2$ is strictly decreasing on $[0,T)$.\par
Moreover, the map $t\mapsto\|u(t)\|_2$ is strictly increasing (resp.\ decreasing) whenever $u(t)\in \mathcal{N}_-$ (resp.\ $u(t)\in \mathcal{N}_+$).\par
Finally, the map $t\mapsto\|u(t)\|_2$ is differentiable and
$$\left(\frac{d}{dt}\|u(t)\|_2\right)^2\le-\frac{d}{dt}J(u(t))\ .$$
\end{proposition}
\begin{proof}
In view of the definition of $\mathcal{N}_\pm$ and Lemma \ref{twonehari}, \eqref{obtain} proves directly statements (ii) and
(iii) and the corresponding strict monotonicity of the map $t\mapsto\|u(t)\|_2$ whenever $u(t)\in \mathcal{N}_\pm$.\par
On the other hand, by Poincar\'e inequality, \eqref{obtain} yields
$$\frac{d}{dt}\|u(t)\|_2^2=-6J(u(t))+\|u(t)\|^2\ge-6J(u(t))+\lambda_1\|u(t)\|_2^2=:\psi(t)\ .$$
By the assumption in (i) we infer that $\psi(0)>0$ so that the map $t\mapsto\|u(t)\|_2$ is initially strictly increasing, say on some
maximal interval $(0,\delta)$ where $\delta>0$ is the first time where $\psi(\delta)=0$. If such $\delta$ exists then, by
Lemma \ref{potentialwell}, also $t\mapsto\psi(t)$ is strictly increasing on $(0,\delta)$ so that $\psi(\delta)>\psi(0)>0$, contradiction.
Therefore $\delta$ does not exist and the maximal interval of strict monotonicity for $t\mapsto\|u(t)\|_2$ coincides with $(0,T)$.\par
Finally, the differentiability of the map $t\mapsto\|u(t)\|_2$ and the estimate of its derivative follows from Lemma \ref{cauchy}.
\end{proof}

\subsubsection*{Finite time blow-up for Navier boundary conditions.}

Consider the initial-boundary value problem
\begin{equation}\label{naviernf}
\left\{\begin{array}{ll}
u_t = \det (D^2 u) - \Delta^2 u \qquad & \text{in }\Omega\times(0,T), \\
u(x,t) = \Delta u(x,t) =0 \qquad & \text{on } \partial \Omega\times(0,T), \\
u(x,0) = u_0(x) \qquad & \text{in }\Omega.
\end{array}\right.\end{equation}
We will prove that the solution to it blows up in finite time provided $u_0$ is large enough in a sense to be specified in the following.
For simplicity we focus on the radial problem set on the unit ball, $\Omega=B_1(0)$, so problem~(\ref{naviernf}) simplifies to
\begin{equation}\label{naviernfr}
\left\{\begin{array}{ll}
u_t = \frac{u_r u_{rr}}{r} - \Delta_r^2 u \qquad & \text{for }r\in[0,1)\, ,\ t>0, \\
u(1,t)=\Delta_r u(1,t) =0 \qquad & \text{for }t>0, \\
u(r,0) = u_0(r) \qquad & \text{for }r\in[0,1)
\end{array}\right.\end{equation}
where $u=u(r,t)$ and $\Delta_r(\cdot) =\frac{1}{r}[r(\cdot)_r]_r$ is the radial Laplacian. Note that
smoothness of the solution implies the symmetry condition $u_r(0,t)=0$ for all $t\ge0$ during the lapse of existence.

\begin{theorem}\label{cease}
Let $u=u(r,t)$ be a smooth solution to~\eqref{naviernfr}. If
$$
\int_0^1 \left( \frac{4}{5}r^5 -\frac{9}{4}r^4 +\frac{5}{2}r^2 \right)(u_0)_r \, dr
$$
is large enough,
then there exists a $T^* <\infty$ such that $u$ ceases to exist when $t \to T^*$.
\end{theorem}
\begin{proof} We begin our proof with the following identity
\begin{equation}\nonumber
\int_0^1 \left( \frac{4}{5}r^5 -\frac{9}{4}r^4 +\frac{5}{2}r^2 \right)u_r \, dr =
- \int_0^1 (4r^4 -9r^3 + 5r)u \, dr,
\end{equation}
where the integration by parts made use of the boundary condition $u(1,t)=0$ and the fact that
one of the roots of the polynomial inside the left hand side integral is located at the origin.
Now, using equation~(\ref{naviernfr}) we get
\begin{eqnarray}\nonumber
-\frac{d}{dt} \int_0^1 (4r^4 -9r^3 + 5r)u \, dr &=& -\int_0^1 (4r^3 -9r^2 + 5) u_r u_{rr} \, dr \\ \nonumber
\ & & +\int_0^1 (4r^3 -9r^2 + 5) [r(\Delta_r u)_r]_r \, dr.
\end{eqnarray}
The first integral on the right hand side can be estimated integrating by parts
\begin{equation}\nonumber
-\int_0^1 (4r^3 -9r^2 + 5) u_r u_{rr} \, dr = \int_0^1 (6r-9) (u_r)^2 \, r \, dr,
\end{equation}
where we have used the symmetry condition $u_r(0,t)=0$ and the fact that the polynomial inside the integral
on the left hand side has one root at $r=1$. We now estimate the integral
\begin{eqnarray}\nonumber
\int_0^1 (4r^3 -9r^2 + 5) [r(\Delta_r u)_r]_r \, dr &=& -\int_0^1 (12r^2 -18r) (\Delta_r u)_r \, r \, dr\\ \nonumber
&=& \int_0^1 (36r^2 -36r) \Delta_r u \, dr\\ \nonumber
&=& -36 \int_0^1 u_r \, r \, dr,
\end{eqnarray}
where the boundary terms vanish due to the presence of roots of the polynomial at the boundary points in the first case,
due to the root of the polynomial at the origin and the boundary condition $\Delta_r u(1,t)=0$ in the second case and
due to the roots of the polynomial and the symmetry condition $u_r(0,t)=0$ in the third case.

Summarizing we have
\begin{equation}\nonumber
\frac{d}{dt} \int_0^1 \left( \frac{4}{5}r^5 -\frac{9}{4}r^4 +\frac{5}{2}r^2 \right)u_r \, dr =
\int_0^1 (6r-9) (u_r)^2 \, r \, dr - 36 \int_0^1 u_r \, r \, dr.
\end{equation}
Therefore
\begin{eqnarray}\nonumber
\frac{d}{dt} \int_0^1 \left( \frac{4}{5}r^5 -\frac{9}{4}r^4 +\frac{5}{2}r^2 \right)u_r \, dr &\le& - \int_0^1 (9-6r) (u_r)^2 \, r \, dr \\ \nonumber
& & +36 \frac{C^2}{2\epsilon} + 36 \frac{\epsilon}{2} \int_0^1 (9-6r) (u_r)^2 \, r \, dr \\ \nonumber
&\le & -C \int_0^1 (9-6r) (u_r)^2 \, r \, dr + C',
\end{eqnarray}
where we have used
\begin{eqnarray}\nonumber
\int_0^1 u_r \, r \, dr &=& \int_0^1 \frac{\sqrt{9-6r}}{\sqrt{9-6r}} u_r \, r \, dr \\ \nonumber
&\le& \left( \int_0^1 \frac{1}{9-6r} r \, dr \right)^{1/2} \left( \int_0^1 (9-6r) (u_r)^2 \, r \, dr \right)^{1/2} \\ \nonumber
&\le& C \left( \int_0^1 (9-6r) (u_r)^2 \, r \, dr \right)^{1/2} \\ \nonumber
&\le& \frac{C^2}{2 \epsilon} + \frac{\epsilon}{2} \left( \int_0^1 (9-6r) (u_r)^2 \, r \, dr \right),
\end{eqnarray}
and here we have employed H\"{o}lder inequality and Young inequality in the first and third inequalities respectively.

We finish our proof with the estimate
\begin{eqnarray}\nonumber
\frac{d}{dt} \int_0^1 \left( \frac{4}{5}r^4 -\frac{9}{4}r^3 +\frac{5}{2}r \right)u_r \, r \, dr \le \\ \nonumber
- C \int_0^1 \left( \frac{4}{5}r^4 -\frac{9}{4}r^3 +\frac{5}{2}r \right)^2 (u_r)^2 \, r \, dr +C' \le \\ \nonumber
- C \left[ \int_0^1 \left( \frac{4}{5}r^4 -\frac{9}{4}r^3 +\frac{5}{2}r \right)u_r \, r \, dr \right]^2 +C',
\end{eqnarray}
where we have used that $9-6r$ is bounded from below by a positive constant and $4r^4/5 - 9r^3/4 + 5r/2$
is non-negative and bounded from above in $[0,1]$ in the first step and Jensen inequality in the second step.
This automatically implies blow-up in finite time
\begin{equation}\label{L1}
\int_0^1 \left( \frac{4}{5}r^4 -\frac{9}{4}r^3 +\frac{5}{2}r \right)u_r \, r \, dr \to - \infty \quad \text{when} \quad t \to T^{**},
\end{equation}
for $T^{**}< \infty$ and a sufficiently large initial condition. In turn, this proves that the solution ceases to exist at some time $T^{*} \le T^{**}$.
\end{proof}

\begin{remark}\label{smoothornot}
A subtle distinction should be made between solutions which ``cease to exist'' and solutions which ``blow up''. The former concerns the existence
of the smooth solution, the latter concerns the unboundedness of some norm; whence the latter implies the former. Theorem \ref{cease} merely states that the
smooth solution ceases to exist, with no statement about blow-up. In particular, if \eq{L1} held then an integration by parts would
show that the $L^1(B_1(0))$ norm of the solution would blow up.
\end{remark}

We conclude this paper with some natural questions and some open problems.\par\medskip\noindent
$\bullet$ {\bf Uniqueness and/or multiplicity of stationary solutions.}\par
By \cite{n1} we know that \eq{eqtnstat} admits the trivial solution $u\equiv0$ and also a mountain pass solution $\widetilde{u}$.
One can then wonder whether \eq{eqtnstat} also admits further solutions. Note first that the functional $J$ in \eq{J} is not even
and, therefore, standard multiplicity results are not available. In particular, $-\widetilde{u}$ is not a solution to \eq{eqtnstat}.
Does the multiplicity of solutions depend on the domain? What about radial solutions in the ball?
In this case, one can refer to some results in \cite{n2,n3}.\par\medskip\noindent
$\bullet$ {\bf Blow up in $L^p$ norms.}\par
{From} Theorems \ref{differentblowup} and \ref{finiteblup} we learn that when blow up occurs, then also the $W^{1,4}_0(\Omega)$-norm blows up.
What about the $L^p(\Omega)$-norms? Is there some critical exponent $q\in(1,\infty)$ such that the blow up in the $L^p(\Omega)$-norm
occurs if and only if $p>q$? Or does the $L^\infty(\Omega)$-norm remain bounded? And, even more interesting, what happens under Navier boundary conditions?
In this respect, see Theorem \ref{cease} and Remark \ref{smoothornot}.\par\medskip\noindent
$\bullet$ {\bf Qualitative properties of solutions.}\par
It is well-known that the biharmonic operator under Dirichlet boundary conditions
does not satisfy the positivity-preserving property in general domains, see e.g.\ \cite{ggs}. Moreover, also the biharmonic heat operator
in $\mathbb{R}^n$ does not preserve positivity and exhibits only {\em eventual local positivity}, see \cite{fgg,gg} where also nonlinear
problems are considered. For there reasons, a full positivity-preserving property for \eqref{parabolicpde} (such as $u_0\ge0$ implies $u(t)\ge0$)
cannot be expected. However, one can wonder whether \eqref{parabolicpde} has some weaker form of positivity-preserving, for instance bounds for
the negative part of $u(t)$ when $u_0\ge0$.\par\medskip\noindent
$\bullet$ {\bf Other boundary conditions.}\par
According to the physical model one wishes to describe, it could be of interest to study \eqref{parabolicpde} with different boundary conditions.
In particular, it could be interesting to consider in more detail the Navier boundary conditions $u=\Delta u=0$ on $\partial\Omega$.
For the stationary problem \eq{eqtnstat}, these conditions were studied in \cite{n2,n3,n1}. It turns out that \eq{eqtnstat} is no longer of
variational type and different techniques (such as fixed point theorems) need to be employed. Therefore, it is not clear whether an energy
functional can be defined and if the same proofs of the present paper may be applied. More generally, one could also consider
the so-called Steklov boundary conditions $u=\Delta u-au_\nu=0$ on $\partial\Omega$,
where $a\in C(\partial\Omega)$ should take into account
the mean curvature of the (smooth) boundary. We refer to \cite{ggs} for the derivation and the physical meaning of these conditions.\par\medskip\noindent
$\bullet$ {\bf Further regularity of the solution.}\par
Using the regularizing effect of the biharmonic heat operator, one could wonder which (maximal) regularity should be expected for solutions to \eqref{parabolicpde}.\par\medskip\noindent
$\bullet$ {\bf High energy initial data.}\par
Except for Theorem \ref{highenergy}, in order to prove global existence or finite time blow-up for (\ref{parabolicpde}) we assumed that $J(u_0)\le d$.
What happens for $J(u_0)>d$? Possible hints may be
found in \cite{gaz,gawe} although the lack of a comparison principle for (\ref{eqtnstat}) certainly creates more difficulties. Can the basin of attraction of the
trivial solution $u\equiv0$ be characterized more explicitly?\par\medskip\noindent
$\bullet$ {\bf Higher space dimensions.}\par
If we set the equation \eqref{parabolicpde} in some $\Omega\subset\mathbb{R}^n$ with $n\ge3$ we lose the physical application but the problem is mathematically challenging.
If $n\le4$ the embedding $W^{2,2}_0(\Omega)\subset W^{1,4}_0(\Omega)$ is still true, although for $n=4$ it becomes a critical embedding which lacks compactness.
Moreover, the embedding $W^{2,2}_0(\Omega)\subset L^\infty(\Omega)$ fails for $n\ge4$. But the most relevant problem concerns the nonlinearity $\detu$ which has the
same degree as the dimension. For instance, if $n=3$ the term $\detu$ is cubic, involving products of three second order derivatives. Since each derivative
merely belongs to $L^2(\Omega)$ (whenever $u\in W^{2,2}_0(\Omega)$), this term may not belong to any $L^p$ space. Hence, no variational approach can be used and a different
notion of solution is needed.

\vfill\eject
\noindent
{\footnotesize
Carlos Escudero\par\noindent
Departamento de Matem\'aticas\par\noindent
Universidad Aut\'onoma de Madrid\par\noindent
{\tt carlos.escudero@uam.es}\par\vskip1mm\noindent
Filippo Gazzola\par\noindent
Dipartimento di Matematica\par\noindent
Politecnico di Milano\par\noindent
{\tt filippo.gazzola@polimi.it}\par\vskip1mm\noindent
Ireneo Peral\par\noindent
Departamento de Matem\'aticas\par\noindent
Universidad Aut\'onoma de Madrid\par\noindent
{\tt ireneo.peral@uam.es}\par\vskip1mm\noindent}
\end{document}